%% file: SNA.tex
\documentclass[leqno,10pt]{article}
\usepackage{german, amsmath, amssymb, amsthm, latexsym, fancyhdr, enumerate, txfonts, expdlist, hyperref, mdwlist, mathrsfs, tikz, todonotes, empheq}
\usepackage[active]{srcltx}
\usepackage[T1]{fontenc}
\usepackage[latin1]{inputenc}
\usepackage{subfig}
\input{bibstandard}

\theoremstyle{plain} 
\newtheorem{thm}{Theorem}[section]
\newtheorem{lem}[thm]{Lemma}
\newtheorem{prop}[thm]{Proposition}
\newtheorem{cor}[thm]{Corollary}

\newtheorem*{thm*}{Theorem}

\theoremstyle{definition}
\newtheorem{defn}[thm]{Definition}

\theoremstyle{remark}
\newtheorem{rem}[thm]{Remark}

\numberwithin{equation}{section}
\numberwithin{thm}{section}

\newcounter{mycount}

\makeatletter
\newcommand*{\rom}[1]{\expandafter\@slowromancap\romannumeral #1@}
\makeatother

\title{Non-Smooth Saddle-Node Bifurcations of Forced Monotone Interval Maps \rom{1}: Existence of an SNA}
\author{Gabriel Fuhrmann}
\date{}
\begin{document}
\maketitle
\begin{abstract}
We study one-parameter families of quasi-periodically forced monotone interval maps and 
provide sufficient conditions
for the existence of a parameter at which the respective system
possesses 
a non-uniformly hyperbolic attractor. This is equivalent to the existence of a sink-source orbit, that is, an orbit with positive Lyapunov exponent both forwards and backwards in time.
The attractor itself is a non-continuous invariant graph with negative Lyapunov exponent, often referred to as ``SNA''.
In contrast to former results in this direction,
our conditions are $\mc C^2$-open in the fibre maps.

By applying a general result about saddle-node bifurcations in skew products, we obtain a conclusion on the occurrence of non-smooth bifurcations in the respective families.
Explicit examples show the applicability of the derived statements.
\end{abstract}

\setlength{\oddsidemargin}{-0.03\textwidth}
\section{Introduction}
Bifurcation theory investigates qualitative changes in the long-term behaviour of a dynamical system along a continuous variation of the system.
For the simple case of a monotonously increasing interval map, the dynamics are qualitatively understood if the fixed points of the respective system are known. 
Hence, in this case bifurcation theory investigates the bifurcation of fixed points.
A well-known example is the saddle-node bifurcation of a one-parameter family of concave functions: If the considered parameter is ``small'', we have one attracting and one repelling fixed point, which approach each other upon the growth of the parameter until some threshold is reached. Above this threshold, these two fixed points have vanished. At the threshold itself, the two points merge together to one neutral fixed point.

In general, non-autonomous systems ask for other objects than fixed points (which might not even exist) to describe the qualitative dynamics \cite{Pötsche}. In the context of quasi-periodically forced monotone maps, 
a natural choice are \emph{invariant graphs}. Like fixed points of monotone interval maps, these are barriers which can't be crossed by an orbit. Furthermore, there is a one-to-one correspondence between the ergodic measures of a quasi-periodically forced monotone map and its invariant graphs (cf.\cite[Theorem~1.8.4]{Arnold}, \cite[Theorem~2.2]{Anagnostopoulou}).

We consider bifurcations of the invariant graphs of a class of quasi-periodically driven (also: forced) monotone interval maps 
\begin{align}\label{eq: skew-product introduction}
f\: \T^d \times \X \to \T^d \times \X, \quad (\theta,x) \mapsto (\theta+\w,\tilde f(\theta,x)),
\end{align}
where $\X$ denotes an interval, $\w$ is irrational,
$\tilde f(\theta,\cdot)$ is strictly increasing, and $\tilde f$ is $\mc C^2$. 
The most studied bifurcation patterns related to such maps are the pitchfork bifurcation (cf. \cite{JägerAMS, Glendinning, Feudel}) and the saddle-node bifurcation (cf. \cite{JägerAMS, Anagnostopoulou, Bjerklöv}). In this article, we deal with the latter phenomenon: 
Consider a one-parameter family of driven interval maps of the form 
(\ref{eq: skew-product introduction}) and
assume that for small parameters there are two invariant graphs (an attracting and a repelling one) which approach each other along the growth of the parameter until some threshold is reached. Above this threshold, these two invariant graphs have vanished. 

In contrast to the autonomous situation, the forced case allows a dichotomy at the threshold:
Either, there is just one neutral invariant graph.
Or, there are two invariant graphs; an attracting and a repelling one. Further, these graphs are \emph{pinched}, that is,
they coincide in a point, and hence on a residual set \cite{Stark}, while they almost surely differ from each other (cf. Figure~\ref{fig: non-smooth bifurcation}). This is called a \emph{non-smooth saddle-node bifurcation}. The attracting graph is referred to as a \emph{strange non-chaotic attractor} (SNA); the repelling one as a \emph{strange non-chaotic repeller} (SNR).

Quasi-periodic forcing and SNAs play an important role in a large class of models for real life systems: The Harper map is a mathematically well-understood dynamical system related to a certain kind of quasi-periodic Schr{\"o}dinger equations (see below); there is numerical evidence for the existence of SNAs in the physiologically relevant Izhikevich Neuron Model \cite{Kim}; \cite{Sonechkin} motivates that not just in order to get a complete description of the tides -as the result of the gravitational interaction between the Earth, the Moon, and the Sun- but even to predict interdecadal atmospheric variations, strange non-chaotic attractors have to be considered. Further, 
\cite{Crucifix} investigates the succession of ice ages and numerically encounters 
bifurcation phenomena creating SNAs.

It is thus desirable to understand the underlying principles of the creation of SNAs.
First results in this direction were obtained by Million\u{s}\u{c}ikov \cite{Millionscikov}, Vinograd \cite{Vinograd} and Herman \cite{Herman} who considered quasi-periodic $\text{SL}(2,\R)$-cocycles. In this context, the phenomenon is also known under the name of \emph{non-uniform hyperbolicity}. In 1984, Grebogi et al. found numerical evidence for SNAs in so-called \emph{pinched 
skew-products} \cite{GOPY}; a rigorous proof in this setting is due to Keller \cite{Keller}. 
Still, these findings lack some flexibility. 
By implementing parameter exclusion techniques,
this has been overcome in more recent results by Young \cite{Young}, Bjerkl\"ov \cite{Bjerklöv} and J\"ager \cite{Jäger}.

Inspired by these works, it is the goal of this article to derive
general conditions for the existence of SNAs, and thus for the occurrence of non-smooth saddle-node bifurcations for maps of the form (\ref{eq: skew-product introduction}).
So far, related results have only been dealing with special families of skew-products \cite{Bjerklöv, JägerAMS}.
While \cite{Bjerklöv} considers the particular case of the Harper map, \cite{JägerAMS} yields a result for additive forcing by imposing a non-differentiability assumption on the forcing term and therefore excludes the application to smooth examples.
The main achievement of this article is the merging of these two technically demanding approaches in order to get a natural and flexible statement on the non-smoothness of saddle-node bifurcations.
The following assertion is a direct consequence of our results.

\begin{thm*}
Let $\X$ be an interval, suppose $\w \in \T^d$ is Diophantine and consider the space
of one-parameter families
\begin{align*}
\mc F_\w=\left\{\left(f_\beta\right)_{\beta \in[0,1]}\: f_\beta \text{ is of the form (\ref{eq: skew-product introduction}) with rotation number } \w \text{ and  phase space } \X \text{; }  f_\beta \text { is }\mc C^1 \text{ in } \beta \right\}
\end{align*} 
equipped with the metric
\begin{align*}
d\left(\left(f_\beta\right)_{\beta \in[0,1]}, \left(g_\beta\right)_{\beta \in[0,1]}\right)=\sup_{\beta\in[0,1]} \left(\left\|f_\beta-g_\beta \right\|_2+\left\|\d_\beta f_\beta-\d_\beta g_\beta\right\|_0\right).
\end{align*}
There exists an open set $\mc U\ssq \mc F_\w$ such that each $\left(f_\beta\right)_{\beta\in[0,1]}\in \mc U$ undergoes a non-smooth saddle-node bifurcation.
\end{thm*}

\noindent
The precise description of the set $\mc U$ by means of conditions on the \emph{maps} $\tilde f_\beta$ is given in Section~\ref{sect: assumptions on the fibre maps}. Though a bit technical, these conditions are explicit and easy to check in examples.
In order to demonstrate the applicability, we further show that for each Diophantine $\w \in \T^1$ the family
$\left(f_\beta\right)_{\beta \in [0,1]}$
\begin{align*}
f_\beta\:  \T^1 \times \R \to \T^1 \times \R, \quad
(\theta,x)   \mapsto \left(\theta+\w, \arctan(\alpha x)-\beta (1+\cos2\pi\theta)\right)
\end{align*}
undergoes a non-smooth saddle-node bifurcation if $\alpha$ is big enough, by showing
that $\left(f_{\beta}\right)_{\beta\in [0,1]}$ lies in the open set $\mc U\ssq \mc F_\w$ (cf. Figure\ref{fig: non-smooth bifurcation}). 

\begin{figure}
\centering
\subfloat[]{\includegraphics[width=2.12in]{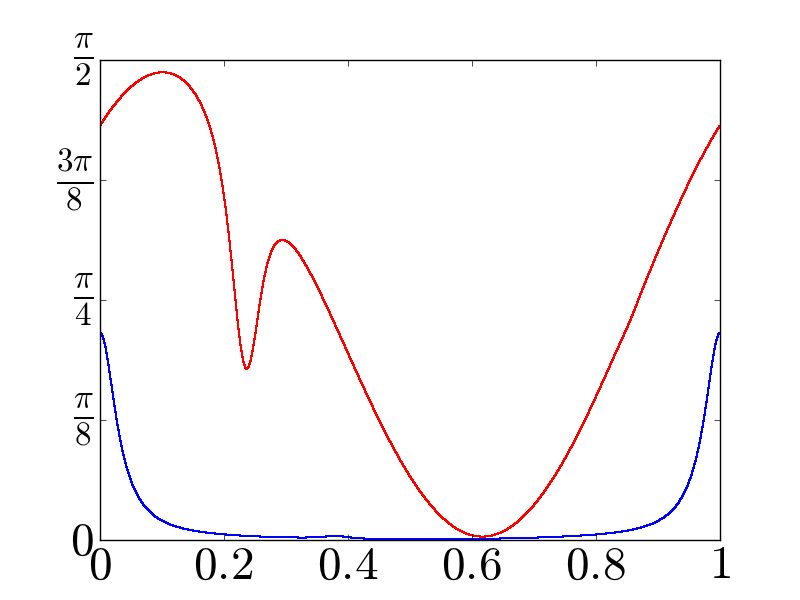}}
\subfloat[]{\includegraphics[width=2.12in]{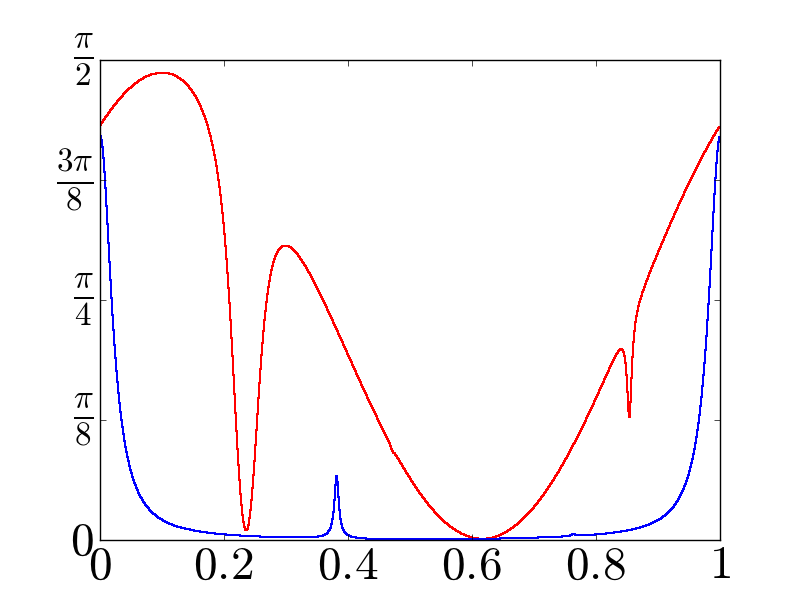}}
\subfloat[]{\includegraphics[width=2.12in]{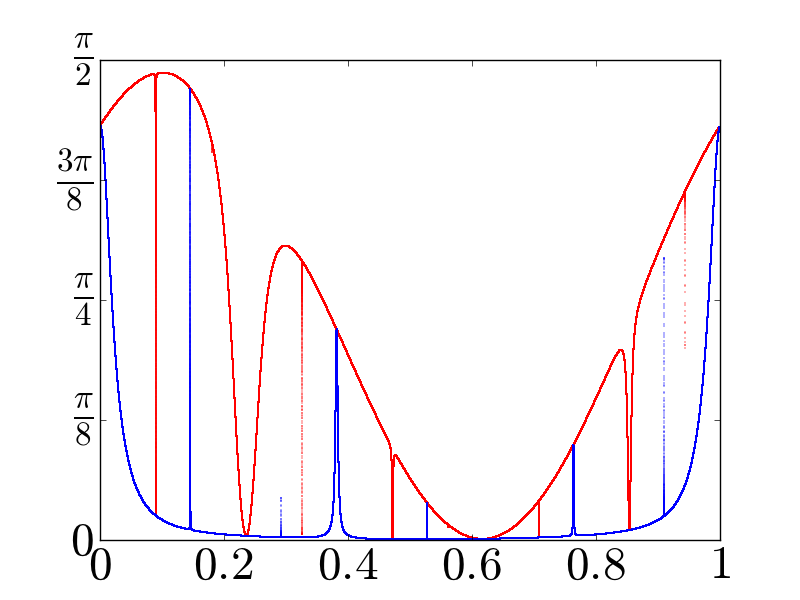}}
\caption{The non-smooth saddle-node bifurcation of the family
$(\theta,x)\mapsto \left(\theta+\w,\arctan(\alpha x)-\beta (1+\cos2\pi \theta)\right)$, where $\w$ is the golden mean and $\alpha=100$. 
We see how the invariant graphs approach each other on a measure zero set as we increase the parameter $\beta$ from the left to the right. The red graph is attracting; the blue one repelling.
(a) $\beta=0.7769$; (b) $\beta=0.7805$;
(c) $\beta=0.7805931$.}
\label{fig: non-smooth bifurcation}
\end{figure}

With the existence of SNAs there naturally arise several questions: At the bifurcation point, the set which is bounded by the pinched graphs is invariant. Is it even minimal? For the Harper map there is a positive answer to this problem for particular parameter regions \cite{Bjerklöv}; the case of \emph{pinched} skew products has been studied in \cite{JägerETDS}.
Further, and closely related: What is the Hausdorff dimension of the pinched graphs \cite{GrögerJäger}? These problems are dealt with in a follow-up article \cite{FuhrmannJäger}.
Finally, we don't address the problem of the speed with which the two initially continuous invariant graphs converge to the SNA and SNR, respectively \cite{BjerklövSaprynka, HaroLlave}. Note that this information is crucial when it comes to the prediction of a bifurcation in real life problems as in \cite{Scheffler}, where the collapse of a bacteria population which is exposed to an increasing light pressure is studied.
 
The present work provides the basis for these further studies and the developed techniques should allow to answer the above questions and approach further problems of similar type.
\paragraph{Acknowledgements.}
I would like to thank Tobias J\"ager for pointing out this problem to me. Moreover, I am very grateful for the fruitful discussions with him and Alejandro Passeggi as well as Maik Gr\"oger. 

This work was supported by the Emmy-Noether-grant ``Low-dimensional and Nonautonomous
Dynamics'' (Ja 1721/2-1) of the German Research Council.
\section{Preliminaries}
\subsection{Basic setting and notation}\label{sect: basic setting and notation}
Throughout this article, we consider families of $\mc C^2$-skew-products\footnote{As the hypothesis of our main result are of a local form, we need 
the systems to be $\mc C^2$ only in a section of the phase space, while on the complement of this section it suffices to assume they are just continuous and continuously differentiable with respect to $x$.} $\left(f_\beta\right)_{\beta \in [0,1]}$ 
\begin{align}\label{eq: skew-product}
\begin{split}
 f_\beta\: \T^d\times \X &\to \T^d\times \X,\\
 (\theta,x)&\mapsto \left(\theta+\w,\tilde f_\beta(\theta,x)\right),
\end{split}
\end{align}
where $\X$ is an interval (possibly the real line, half-open, \ldots),
$\tilde f_\beta(\theta,\cdot)$ is strictly increasing
and $\w \in \T^d$ is \emph{Diophantine} in the following sense.
\begin{defn}
Let $\mathscr C, \eta>0$. We say
$\w\in \T^d$ is \emph{Diophantine} of type $(\mathscr C,\eta)$ if
$d(k\w,0)\geq \mathscr C |k|^{-\eta}$
for all  $k\in\Z \setminus \{0\}$, where $|\cdot|$ denotes the Euclidean distance.
\end{defn}
It is obvious that the above definition includes all those $\w \in \T^d$ which are Diophantine in the usual sense.
We refer to a skew-product of the form (\ref{eq: skew-product}) as a \emph{quasi-periodically forced} (qpf) \emph{monotone} interval map.
$\T^d$ is called the \emph{base} and $\{\theta\}\times\X$ a \emph{fibre}
of (\ref{eq: skew-product}) for $\theta \in \T^d$.

By writing $f_\beta^{-l}(\theta,x)$ for $l\in \N$, we implicitly assume
$(\theta,x) \in f_\beta^l(\T^d,\X)$ such that the respective expression is well-defined due to the injectivity of $f_\beta$.
For $l\in \Z$ and $(\theta_0,x_0) \in \T^d\times \X$, we set $(\theta_l,x_l)\=f^l(\theta_0,x_0)$.
Given $\beta$ and $\theta$, we call 
\begin{align*}
f_{\beta,\theta}\:x \mapsto \tilde f_\beta(\theta,x)
\end{align*} 
a \emph{fibre map} of the skew-product (\ref{eq: skew-product}).
We denote the fibre map of the $n$-th iterate of $f_\beta$ by $f_{\beta,\theta}^n$, where $n\in \Z$. Hence, for $n\in \N$ we have 
\begin{align*}
f_{\beta,\theta}^n(x)=f_{\beta,\theta+(n-1)\w}\circ \ldots \circ f_{\beta,\theta}.
\end{align*}
Further,
$f_{\beta,\theta}^{-1}(x)=(f_{\beta,\theta-\w})^{-1}(x)$.

We denote the derivative of $\pi_2\circ f_{\beta}^n(\theta,x)$ with respect to $\beta$ by $\d_\beta f^n_{\beta,\theta}(x)$, where $\pi_2$ is the canonical projection to the second coordinate; the directional derivative of $\pi_2\circ f_{\beta}^n(\theta,x)$ with respect to a direction $\vartheta \in \T^d\setminus\{0\}$ is denoted by $\d_\vartheta f_{\beta,\theta}(x)$. Higher derivatives are denoted in an analogous way. Typically, we will consider $\vartheta$ to be a unit vector and write $\vartheta \in \mathbb S^{d-1}$. The derivative of the fibre maps are denoted by $\d_x f_{\beta,\theta}(x)$.

Given $\beta \in [0,1]$, we are interested in studying the \emph{invariant graphs} of $f_\beta$. These are measurable functions $\phi\: \T^d\to\X$ satisfying
\begin{align*}
 f_\beta(\phi(\theta))=\phi(\theta+\w),
\end{align*}
such that the corresponding graph $\Phi\=\{(\theta,\phi(\theta))\in \T^d\times \X\:\theta \in \T^d\}$ is an invariant set under $f_\beta$ in the sense
that $f_\beta(\Phi)=\Phi$. Like fixed points in monotone, unforced (that is, autonomous) interval maps, invariant graphs are barriers that can not be crossed by orbits (if we consider monotonously increasing fibre maps, too) and as a matter of fact, many statements for fixed points of unforced maps remain true (with slight changes) when going over to invariant graphs of qpf interval maps.

As in the unforced situation, the stability of an invariant graph $\phi$ is closely related to its associated Lyapunov exponent (\cite[Proposition~3.3]{JägerNonLin}), which is given by
\begin{align*}
 \lam(\phi)\=\int_{\T^d}\! \log\left|\d_xf_\theta(\phi(\theta))\right| \, d\theta.
\end{align*}
We say an invariant graph $\phi$ is an \emph{attractor} if $\lam(\phi)<0$; we call it 
\emph{repeller} if $\lam(\phi)>0$; and we call it \emph{neutral} if $\lam(\phi)=0$.

\begin{thm}[{cf. \cite[Theorem~2.1]{Anagnostopoulou}}]\label{thm: concave fibre maps Keller}
Consider a qpf monotone interval map $f$ of the form (\ref{eq: skew-product}). Assume that for each $\theta \in \T^d$ there exist measurable functions $\gamma^-\leq \gamma^+\:\T^d\to \X$ such that for all $\theta \in \T^d$ the fibre maps are strictly concave on $\Gamma(\theta)\=[\gamma^-(\theta),\gamma^+(\theta)]$. Further, assume that
$h(\theta)\=\inf_{x\in\Gamma(\theta)} \log \d_x f_\theta(x)$ has an integrable minorant.

Then there exist at most two distinct invariant graphs in $\Gamma\=\{(\theta,x)\in \T^d\times \X\:x\in \Gamma(\theta)\}$.\footnote{A graph $\phi$ is said to be contained in $\Gamma$ if $\phi(\theta)\in \Gamma(\theta)$ for all $\theta \in \T^d$. Further, we identify invariant graphs which coincide almost surely.} Moreover, if there exist two distinct invariant graphs $\phi^-\leq \phi^+$ in $\Gamma$, then $\phi^-$ is a repeller and $\phi^+$ is an attractor.
\end{thm}

We call two invariant graphs $\phi,\psi$ pinched if $\phi(\theta)=\psi(\theta)$ for some $\theta$. Note that due to the minimality of the \emph{base map} $\theta\mapsto \theta+\w$, this implies that the two graphs coincide on a residual set in $\T^d$ \cite{Stark}.
It is obvious that in this case at least one of the two graphs is non-continuous. A non-continuous invariant graph $\phi$ is called a \emph{strange non-chaotic attractor} (SNA) if $\lam(\phi)<0$; it is called a \emph{strange non-chaotic repeller} (SNR) if $\lam(\phi)>0$.

In 1984, Grebogi et al. found numerical evidence for the existence
of an SNA in the case of \emph{pinched systems} \cite{GOPY}. These are qpf interval maps $f$ with monotone fibre maps which leave the zero line invariant (that is, $f_\theta(0)=0$) and posses a \emph{pinched point}, that is, there exists a point $\theta_0 \in \T^d$ with $f_{\theta_0}(x)=0$. Under additional assumptions, the zero-line turns out to be repelling. Thus, by proving the existence of an attracting graph (which necessarily has to share a residual set with the zero-line), Keller gave a rigorous argument for the existence of an SNA for pinched systems \cite{Keller}. In the injective setting (where there are no pinched points), we cannot argue in such a comparably direct way. Instead, the following concept proves helpful.
\begin{defn}
A \emph{sink-source orbit} is an orbit
whose backward and forward vertical Lyapunov exponent is positive, that is,
\begin{align*}
 \limsup_{n\to \infty} \frac1n \ln |\d_xf_{\beta,\theta}^n(x)|>0 \quad \text{ and }\quad
\limsup_{n\to \infty} \frac1n \ln |\d_xf_{\beta,\theta}^{-n}(x)|>0.
\end{align*} 
\end{defn}

\begin{thm}\emph{(\cite[Theorem~2.4]{JägerAMS})}\label{thm: sink source orbit <-> SNA}
 Suppose $f_\beta\: \T^d\times [a,b]\to \T^d\times[a,b]$ is a continuous qpf monotone interval map and
$(\theta,x)\mapsto \d_x f_\beta(x)>0$ is continuous. Then the existence of a sink-source-orbit implies the existence of both an SNA and an SNR.
\end{thm}
We prove the existence of an SNA by proving the existence of a sink-source orbit.
Note that if $\X=\R$, we may apply Theorem~\ref{thm: sink source orbit <-> SNA} by considering $\X$ to be the extended real line $[a,b]=[-\infty,\infty]$, since the constructed sink-source orbit is bounded. We deal in an analogous way with open and half-open intervals, respectively.
\subsection{Saddle-node bifurcations in quasi-periodically forced systems}\label{sect: saddle-node bifurcations in qpf systems}
Non-autonomous bifurcation theory in the setting of one-parameter families of qpf monotone interval maps studies the bifurcation of invariant graphs along the change of the parameter.
An often considered situation is that of a saddle-node bifurcation \cite{JägerSiegmund, JägerAMS, Anagnostopoulou}: There exists a critical parameter $\beta_c$ such that for $\beta<\beta_c$ there are two continuous invariant graphs, while
there is no invariant graph for $\beta>\beta_c$. At $\beta=\beta_c$ there exists (in contrast to the un-forced case) a dichotomy:
Either there occurs a \emph{smooth} or a \emph{non-smooth saddle-node} bifurcation.

\begin{thm}\emph{(cf. \cite[Theorem~6.1]{Anagnostopoulou})}\label{thm: anagnostopoulou}
 Let $\w\in \T$ and suppose
$\left(f_\beta \right )_{\beta \in [0,1]}$ is a family of qpf monotone $\mc C^2$ interval maps. Further, assume that there exist continuous functions $\gamma^-, \gamma^+\: \T^d\to \X$ with $\gamma^-<\gamma^+$ such that the following holds (for all $\beta \in [0,1]$ and
$\theta \in \T^d$ where applicable).
\begin{enumerate}[(i)]
 \item There exist two distinct continuous $f_0$-invariant graphs and no 
$f_1$-invariant graph in $\Gamma$;
\item $f_{\beta,\theta}(\gamma^\pm(\theta))\leq \gamma^\pm(\theta+\w)$;
\item the maps $(\beta,\theta,x)\mapsto \d_x^i f_\beta(\theta,x)$ with $i=0,1,2$ and
$(\beta,\theta,x)\mapsto \d_\beta f_\beta(\theta,x)$ are continuous;
\item $\d_x f_{\beta,\theta}(x)>0 $ for all $ x\in\Gamma(\theta)$ \label{bull: bifurcation 5};
\item $\d_x^2 f_{\beta,\theta}(x)<0 \ (x\in \mathring \Gamma(\theta))$;
\item $\d_\beta f_{\beta,\theta}(x)<0 \ (x\in \Gamma(\theta))$.
\end{enumerate}
Then there exists a unique critical parameter $\beta_c\in (0,1)$ such that there holds:
\begin{itemize}
 \item If $\beta <\beta_c$, then there exist two continuous $f_\beta$-invariant graphs 
$\phi_\beta^-<\phi_\beta^+$ in $\Gamma$ with $\lambda(\phi^-_\beta)<0$ and 
$\lambda(\phi_\beta^+)>0$.
\item If $\beta=\beta_c$, then either there exists exactly one $f_\beta$-invariant 
graph $\phi_\beta$ in $\Gamma$, or there exist two semi-continuous and pinched $f_\beta$-invariant graphs $\phi_\beta^-<\phi_\beta^+$ a.s. in $\Gamma$, with $\phi^-_\beta$ lower and $\phi_\beta^+$ upper semi-continuous. 
If there is only one invariant graph $\phi_\beta$, then $\lambda(\phi_\beta)=0$. 
If there are two graphs, then $\lambda(\phi_\beta^-)>0$ and
$\lambda(\phi_\beta^+)<0$. 
\item If $\beta>\beta_c$, then no $f_\beta$-invariant graph exists in $\Gamma$.
\end{itemize}
\end{thm}
\begin{rem}
If there exist two invariant graphs at the critical parameter $\beta_c$, we speak of a \emph{non-smooth saddle-node bifurcation}. The other case is referred to as a \emph{smooth} bifurcation.
\end{rem}
The main goal of this article is to provide natural conditions under which the occurrence of a non-smooth saddle-node bifurcation is guaranteed.
\section{Statement of the main result and applications}\label{sect: assumptions on the fibre maps}
We first collect a number of assumptions on the considered skew-products which we need in order to formulate our main result. In order to both make the reader familiar with these assumptions and to demonstrate how they apply to some standard skew-product families, we explicitly show that they are satisfied by
\begin{align}\label{eq: arctan}
\tag{$\ast$}
\begin{split}
f_\beta\:  \T^1 \times \R &\to \T^1 \times \R\\
(\theta,x)  & \mapsto \left(\theta+\w, \arctan(\alpha x)-\beta\cdot \frac{\pi}4 (1+\cos2\pi\theta)\right),
\end{split}
\end{align}
for Diophantine $\w \in \T^1$ and large enough $\alpha$.

To guarantee the existence of a sink source orbit, we need 
to ensure that the respective orbit spends most of the positive times
in regions of (vertical) expansion and most of the negative times in 
(vertically) contracting regions. For that reason, we assume the existence
of both an interval of expansion
$E=[e^-,e^+]$ and contraction $C=[c^-,c^+]$ with $e^+<c^-$ and such that
\begin{enumerate}[${(\mathcal A}1)$]
\item $\d_x f_{\beta,\theta}(x)<\alpha_c$ for $(\theta,x)\in \T^d\times C$,\label{axiom: 1}
\item $\d_x f_{\beta,\theta}(x)>\alpha_e$ for $(\theta,x)\in \T^d\times E$,\label{axiom: 2}
\suspend{enumerate}
where $0<\alpha_l<\alpha_c<1<\alpha_e<\alpha_u$ and
\resume{enumerate}[{[${(\mathcal A}1)$]}]
\item $\alpha_l<\d_x f_{\beta,\theta}(x)<\alpha_u$ for all $(\theta,x)\in \T^d\times [e^-,c^+]$
\label{axiom: 3}. 
\suspend{enumerate}
Instead of considering all of the phase space $\T^d\times \X$, we restrict our analysis to the section $\T^d\times [e^-,c^+]$. 
Thus, $e^-$ and $c^+$ play the roles of $\gamma^\mp$ in Theorem~\ref{thm: anagnostopoulou}.
\resume{enumerate}[{[${(\mathcal A}1)$]}]
\item $f_{\beta,\theta}(c^+)\leq c^+$ and $f_{\beta,\theta}(e^-)\leq e^-$. \label{axiom: 4}
\suspend{enumerate}
In Theorem~\ref{thm: anagnostopoulou}, there exist two invariant graphs between $\gamma^-$ and $\gamma^+$ for $\beta=0$ and no invariant graphs for $\beta=1$. 
In order to ensure this, we suppose, in addition to $(\mathcal A\ref{axiom: 4})$, that
\resume{enumerate}[{[${(\mathcal A}1)$]}]
\item $f_{0,\theta}(c^-)\geq c^-$ for all $\theta\in \T^d$ and 
$f_{1,\theta}(c^+) \leq e^-$ for some $\theta \in \T^d$. \label{axiom: 5}
\suspend{enumerate}
Before formulating further assumptions, let us define the introduced
quantities for (\ref{eq: arctan}) and see how ${(\mathcal A}\ref{axiom: 1})$-
 ${(\mathcal A}\ref{axiom: 5})$ are verified in this particular case.
Set $e^-\=0$, 
$e^+\=r/\alpha$ for some $r>0$, $c^+\=\pi/2$ and fix an arbitrary $c^-$ in $(e^+,c^+)$.
As
\begin{align*}
 &\d_x f_{\beta,\theta}(x)=\frac{\alpha}{1+(\alpha x)^2},
\end{align*}
we get $(\A\ref{axiom: 1})-(\A\ref{axiom: 3})$ with $\alpha_e,\alpha_c^{-1}=\alpha^{\frac2p}$ and
$\alpha_u,\alpha_l^{-1}=\alpha^p$ for some fixed
$p>2$ if $\alpha$ is large enough.
Further, $(\A\ref{axiom: 4})$ is evident and $(\A\ref{axiom: 5})$ trivially holds under the assumption of large enough $\alpha$.

As in Theorem~\ref{thm: anagnostopoulou}, we naturally assume monotone dependence
on $\beta$.
\resume{enumerate}[{[${(\mathcal A}1)$]}]
\item $f_{(\cdot)}(\theta,\x)$ is strictly decreasing for fixed $(\theta,x)\in \T^d\times [e^-,c^+]$. \label{axiom: 6}
\suspend{enumerate}
Furthermore, we need the dependence on $\beta$ to be smooth enough, that is, we suppose
\resume{enumerate}[{[${(\mathcal A}1)$]}]
\item $(\beta,\theta,x)\mapsto f_{\beta}(\theta,x)$ and $(\beta,\theta,x)\mapsto\d_x f_{\beta}(\theta,x)$ as well as 
$(\beta,\theta,x)\mapsto\d_\vartheta f_{\beta}(\theta,x)$ for $\vartheta \in \mathbb S^{d-1}$ are continuous. \label{axiom: 7} 
\suspend{enumerate}
Both assumptions are trivially fulfilled by (\ref{eq: arctan}).
As we want to restrict ourselves to $\T^d\times [e^-,c^+]$, we exclude parameters $\beta$ for which obviously every orbit leaves $\T^d\times [e^-,c^+]$. In other words, we only consider parameters not bigger than
$\beta_+(0)\= \min\left\{\beta \in [0,1]\left|\exists \theta \in 
\T^1 \: f_{\beta,\theta}(c^+) = e^-\right. \right\}$.
On the other hand, as we want the sink-source orbit to basically stay in the contracting region $\T^d\times C$ for negative times while we want it in the expanding region $\T^d\times E$ for positive times, we need to ensure that there is a connection between the two regions. Therefore, we only consider parameters $\beta$ not too small in order to make it possible to jump from one region to the other. That is, we deal with parameters not smaller than
$\beta_-(0)\=\max \left\{\beta \in [0,\beta_+(0)]\left| \forall \theta\in \T^1 \: f_{\beta,\theta}(c^-)\geq e^+ \right.\right\}$. Note that we don't need to compute
$\beta_\pm(0)$ for (\ref{eq: arctan}) explicitly. Instead, it suffices to know that $\beta_\pm(0)\in (0,1)$, which is true for obvious reasons.

Setting $\mathscr B(0)\=[\beta_-(0), \beta_+(0)]$, we hence only consider $\beta \in \mathscr B(0)$ from now on.
For each such $\beta$ there is a so-called (first) \emph{critical region}, $\mc I_{0,\beta} \ssq \T^d$ such that outside of $\I_{0,\beta}$, orbits in the contracting region stay in the contracting region.
\resume{enumerate}[{[${(\mathcal A}1)$]}]
 \item $f_{\beta,\theta}\left(x\right)\in C$ for all $x\in [e^+,c^+],\theta \notin 
\I_{0,\beta}$. \label{axiom: 8}
\setcounter{mycount}{\theenumi}
\end{enumerate}
By means of the monotonicity in $(\mc A \ref{axiom: 3})$
and by $(\mc A \ref{axiom: 4})$, this is equivalent to
\begin{enumerate}[${(\mathcal A}1')$]
\setcounter{enumi}{7}
\item $f^{-1}_{\beta,\theta}\left(x\right)\in  E$ for all $x\in [e^-,c^-],\theta \notin \I_{0,\beta}+\w$.\label{axiom: 8 prime}
\end{enumerate}
Notice that a priori we did not assume invertibility of $f_\beta$. However, due to
the inverse function theorem and 
${(\mathcal A}\ref{axiom: 3})$, we have that for a small open neighbourhood
$U$ of $\T^d\times [e^-,c^+]$ the map
$\left({f_\beta}_{|_{U}}\right)^{-1}$ is well-defined and $\mc C^2$. We will refer to it simply as $f_\beta^{-1}$. Observe that it also verifies ${(\mathcal A}\ref{axiom: 7})$.

In general, a natural choice for the critical region is given by
\begin{align*}
\mc I_{0,\beta}\=\left\{\theta\in \T^d\: f_{\beta,\theta}(e^+)\leq c^-\right\},
\end{align*}
which verifies ${(\mathcal A}\ref{axiom: 8})$ by definition.
In the particular case of (\ref{eq: arctan}), this choice reads
\begin{align}\label{eq: I0 arctan}
 \mc I_{0,\beta}\=\left\{\theta \in \T^1\: \cos 2\pi\theta\geq
\frac4\pi\cdot(\arctan r-c^-)/\beta-1 \right\}.
\end{align}

The critical region $\I_{0,\beta}$
allows jumps from the contracting to the expanding region and vice versa.
On the other hand, we also want the sink-source orbit to spend long times in the respective regions without jumping out too often, that is, we don't want $\I_{0,\beta}$ to be too big.

In (\ref{eq: I0 arctan}), we see that by choosing large $r$ and small $c^-$, we can make $\I_{0,\beta}$ arbitrarily small for large enough $\alpha$. This results from the fact that the second derivative
$\d_\theta^2f_{\beta,\theta}(x)=\beta \pi^3 \cdot \cos 2\pi \theta$ is bounded away from $0$ on the \emph{interval} $\I_{0,\beta}$.
In general, we thus assume there exists $s>0$ such that
\begin{enumerate}[${(\mathcal A}1)$]
\setcounter{enumi}{\themycount}
\item $\d_\vartheta^2 f_{\beta,\theta}(x)>s$ for each $\vartheta \in {\mathbb S}^{d-1}$ and $\theta \in \I_{0,\beta}, x\in C, \beta \in \mathscr B(0)$,
\label{axiom: 9}
\item $\I_{0,\beta}$ is closed and \emph{convex} and $\I_{0,\beta}\ssq \I_{0,\beta'}$ for $\beta\leq \beta'$. \label{axiom: 10}
\suspend{enumerate}
To motivate further assumptions, we need to provide a rough sketch of how to prove the existence of a sink-source orbit.
Assuming that $\I_{0,\beta}$ is small, there is a positive number
$M_0$ such that the first $M_0$ forward and backward iterates of $\I_{0,\beta}+\w$ under the base transformation (that is, under the rigid rotation with rotation vector $\w$) don't intersect, that is,
\begin{align*}
 \I_{0,\beta}+\w \cap \bigcup_{k=\pm 1,\ldots,\pm M_0} \left(\I_{0,\beta}+(k+1)\w\right)=\emptyset.
\end{align*}
If this is true,
$f^{l}_\beta(\theta,x)$ never leaves the contracting region for $\theta \in \I_{0,\beta}-(M_0-1)\w, x \in C$ and $l=0,\ldots,M_0-1$, while
$f^{-l}_\beta(\theta,x)$ never leaves the expanding region for
$\theta \in \I_{0,\beta}+(M_0+1)\w, x \in E$ and $l=0,\ldots,M_0$, due to 
${(\mathcal A}\ref{axiom: 8})$ and ${(\mathcal A}\ref{axiom: 8 prime}')$. However, 
$f^{M_0-1}_\beta(\theta,x)$ might jump into the expanding region under the action of $f_\beta$ or, even more, fall into the set $f^{-M_0}_\beta\left(\I_{0,\beta}+(M_0+1)\w,E\right)$. In the latter case,
$f^{M_0-1}_\beta(\theta,x)$ is a first candidate for a sink-source orbit as it stays in the expanding region for $M_0+1$ times while its backward iterates stay in the contracting region for $M_0-1$ times.

The projection of the set of all such sink-source orbit candidates to the base $\T^d$ is denoted by $\I_{1,\beta}$.
Similarly as in the case of $\I_{0,\beta}$, we need that $\I_{1,\beta}$ is small enough to guarantee that it visits itself with an even smaller frequency than $\I_{0,\beta}$. 
To that end, we need that the second derivatives of
$\phi^\pm(\theta)\=f_{\beta,\theta-M_0\w}^{M_0+1}(c^\pm)$ and $\psi^\pm(\theta)=f_{\beta,\theta+M_0\w}^{-M_0-1}(e^\pm)$ (for $\theta \in \I_{0,\beta}+\w$)
with respect to $\theta$ are small in comparison to the lower bound $s$ in 
${(\mathcal A}\ref{axiom: 9})$ (such that the second derivatives of $\phi^\pm$ and $\psi^\pm$ with respect to $\theta$ are basically bounded from below by $s$ as well).
This amounts to keeping all the other derivatives of $f_\beta$ and and its inverse
small. Let $S>0$ be such that
\resume{enumerate}[{[${(\mathcal A}1)$]}]
\item $\left|\d_\vartheta f_{\beta,\theta}(x)\right|<S$
for all $(\theta,x) \in \T^d\times [e^-,c^+] $ and $\vartheta \in \mathbb S^{d-1}
$\label{axiom: 11},
\item $\left|\d_\vartheta^2 f_{\beta,\theta}(x)\right|<S^2$
for all $(\theta,x) \in \T^d\times [e^-,c^+] $ and $\vartheta \in \mathbb S^{d-1}$,
\label{axiom: 12}
\item $ \left|\d_\vartheta\d_x f_{\beta,\theta}(x)\right|<
\begin{cases}
S \alpha_c & \text{for } (\theta,x) \in \T^d\times C \\
S\alpha_u^2 & \text{for } (\theta,x)\in \T^d\times [e^-,c^-)
\end{cases}$ for each $\vartheta \in \mathbb S^{d-1}.$ \label{axiom: 13}
\suspend{enumerate}
Further, suppose
\resume{enumerate}[{[${(\mathcal A}1)$]}]
\item $\left|\d_x^2 f_{\beta,\theta}(x)\right|<
\begin{cases}
\alpha_c & \text{for } (\theta,x)\in \T^d\times C \\
\alpha_u^2 & \text{for } (\theta,x)\in \T^d\times [e^-,c^-)
\end{cases}$. \label{axiom: 14}
\suspend{enumerate}
For the derivatives of the inverse, we get some of the above estimates by means of the inverse function theorem. However, we additionally need
\resume{enumerate}[{[${(\mathcal A}1)$]}]
\item $\left|\d_x^2 f_{\beta,\theta}^{-1}(x)\right|< \alpha_e^{-1}$
for each
$\theta\notin \I_{0,\beta}+\w$ and $x\in E$,\label{axiom: 15}
\item $\left|\d_\vartheta\d_x f_{\beta,\theta}^{-1}(x)\right|< S \alpha_e^{-1}$
for each $\theta\notin \I_{0,\beta}+\w, x\in E$ and $\vartheta \in \mathbb S^{d-1}$.\label{axiom: 16}
\end{enumerate}
Coming back to (\ref{eq: arctan}), we get
${(\mathcal A}\ref{axiom: 11})$ and
${(\mathcal A}\ref{axiom: 12})$ by setting $S\= \max_{\beta,\theta,x} \d_\theta f_{\beta,\theta}(x)=\pi^2/2$. ${(\mathcal A}\ref{axiom: 13})$ is trivial, as mixed derivatives vanish. With
\begin{align*}
 \d_x^2 f_{\beta,\theta}(x)=\frac{-2\alpha^3x}{\left(1+(\alpha x)^2\right)^2},
\end{align*}
we get $\d_x^2 f_{\beta,\theta}(x)< \alpha^{-2/p}$ for big enough $\alpha$ and $x\in C$. Further, basic calculus yields
$\left|\d_x^2 f_{\beta,\theta}(x)\right|\leq \d_x^2 f_{\beta,\theta}\left(\sqrt{1/(3\alpha^2)}\right)=\mc O(\alpha^2)$ as $\alpha\to \infty$. This shows $(\mc A\ref{axiom: 14})$ for big enough $\alpha$. If
$x\in E$ and $\theta \notin \I_0+\w$, we moreover have
\begin{align*}
\d_x^2 f_{\beta,\theta}^{-1}(x)&=
2/\alpha \cdot\frac{\sin\left(x+\beta \cdot \frac\pi4 (\cos2\pi (\theta-\w)+1)\right)}
 {\cos^{3}\left(x+\beta \cdot \frac\pi4 (\cos2\pi (\theta-\w)+1)\right)},\\
\d_\theta \d_x f_{\beta,\theta}^{-1}(x)&=
-\beta \pi^2/\alpha \cdot\frac{\sin2\pi (\theta-\w) \cdot \sin\left(x+\beta \cdot \frac\pi4 (\cos2\pi (\theta-\w)+1)\right)}
 {\cos^{3}\left(x+\beta \cdot \frac\pi4 (\cos2\pi (\theta-\w)+1)\right)}.
\end{align*}
As $\theta \notin \I_{0,\beta}+\w$, (\ref{eq: I0 arctan}) yields
$\beta \cdot \frac\pi4 (\cos2\pi (\theta-\w)+1)<\arctan r -c^-$ which proves $(\A\ref{axiom: 15})$ and $(\A\ref{axiom: 16})$ for large enough $\alpha$, since
$0\leq x\leq r/\alpha$.

We are now in a position to state the main theorem of this article. 
\begin{thm}\label{thm: sink source orbit introduction}
Suppose $\w$ is Diophantine of type $(\mathscr C,\eta)$ and $\left(f_\beta\right)_{\beta\in [0,1]}$ satisfies 
$(\A\ref{axiom: 1})$-$(\A\ref{axiom: 16})$. 
Let there be $p\geq \sqrt{2}, \alpha>1$ with
\begin{align*}
 \alpha_c^{-1}=\alpha_e = \alpha^{\frac2{p}}, \qquad \alpha_l^{-1}=\alpha_u = \alpha^p.
\end{align*}
Then there exist strictly positive constants $\eps_0=\eps_0(p,\mathscr C,\eta)$ and
$\alpha_0=\alpha_0(s,S,p,|C|,|E|,\mathscr C,\eta)$ such that if $|\I_{0,\beta_+(0)}|<\eps_0$ and $\alpha>\alpha_0$, there is $\beta_c\in [0,1]$ such that $f_{\beta_c}$ has a sink-source orbit, and hence an SNA and an SNR in $\T^d\times[e^-,c^+]$.
\end{thm}
Note that together with the previous discussion, this theorem proves the occurrence of a non-smooth bifurcation for the example in the introduction.
\begin{rem}
\begin{enumerate}[(i)]
\item $\alpha_0$ can be chosen to be monotonously increasing in $|C|$ and $|E|$.
\item The conjugacy $(\theta,x)\mapsto (\theta,-x)$ and the parametrisation
$\beta\mapsto 1-\beta$ yield a symmetric version of the theorem if
the contracting region is below the expanding one, that is, if $c^+<e^-$.
\end{enumerate}
\end{rem}
Theorem~\ref{thm: sink source orbit introduction} is proved in Section~\ref{sect: proof of the main theorem} by showing the existence of a sink-source orbit in $\T^d\times[e^-,c^+]$. Due to Theorem~\ref{thm: sink source orbit <-> SNA}, this yields the existence of an SNA and an SNR. Setting $b=c^+$ and changing the system for $x<e^-$ in such a way that  every point below $e^-$ approaches $a$ for $n\to \infty$, we see that the respective SNA/SNR pair is in fact contained in $\T^d\times[e^-,c^+]$.

It is important to note that the assumptions of Theorem~\ref{thm: sink source orbit introduction} are stable under $\mc C^2$-small perturbations of the fibre maps $f_{\beta,\theta}$ which respect $(\A\ref{axiom: 6})$ and $(\A\ref{axiom: 7})$. 
This is the main advantage over previous results in this direction,
which establish the existence of an SNA only under comparably strong technical constraints: In \cite[Theorem~2.7]{JägerAMS} it is necessary to assume the existence of a ``sharp peak'' for the maps $\phi^\pm$, which implies non-differentiability of $f_\beta$ with respect to the base coordinates.\footnote{Note that the positive lower bound for $\d_\vartheta^2 f_{\beta,\theta}(x)$ in $(\A\ref{axiom: 9})$ can be understood as a replacement for this sharp peak assumption  in the respective statement in \cite{JägerAMS}.}

An important step towards the understanding of the creation of SNAs was the verification of a non-smooth saddle-node bifurcation for the Harper map
\begin{align*}
 (\theta,x)\mapsto \left(\theta+\w, \arctan\left(\frac{-1}{\tan(x)-E+\lam 
V(\theta)}\right)\right),
\end{align*} 
which is closely related to the discrete quasi-periodic Schrödinger equation.
In \cite{Bjerklöv} it is shown that
if the ``potential'' $V$ is $\mc C^2$ and if it assumes its unique global maximum at a point with non-vanishing second derivative, then we observe a non-smooth saddle node-bifurcation upon a decrease of $E$ if $\lam$ is large enough.

The geometric idea of our proof is inspired by the proof in \cite{Bjerklöv} as can be 
readily seen from the pictures in Figure~\ref{fig: geometric idea}.
It is thus not surprising that we can recover Bjerklöv's result with the same 
regularity assumptions.\footnote{The application of 
Theorem~\ref{thm: sink source orbit introduction} to the Harper map works by means of a 
similar argument as in \cite[Section~2.4.2]{JägerAMS}.
However, it is necessary to control the 
dependence of $\alpha_0$ on $s$ and $S$ in this particular case.
To that end, we provide a slightly different formulation of the above theorem (cf. Theorem~\ref{thm: sink-source orbit refined}), which specifies the relationship between $\alpha$ and $s$ as well as $S$ in an appropriate way.}
However, as we don't restrict to fibre maps of a particular shape, more work is needed 
in order to get control over the sink-source orbit.

Despite the fact that $(\A\ref{axiom: 1})-(\A\ref{axiom: 16})$ seem rather technical, they just capture the main qualitative properties of some standard examples which posses an SNA and turn out to be flexible enough to treat different skew-product families at the same time. We have seen that (\ref{eq: arctan}) verifies the assumptions of Theorem~\ref{thm: sink source orbit introduction}.
As a generalisation of the $\arctan$-family (\ref{eq: arctan}), for each $q>1$ we can apply Theorem~\ref{thm: sink source orbit introduction} to
\begin{align*}
(\theta,x)  & \mapsto \left(\theta+\w, h_q(\alpha x)-\beta\cdot \frac{h_q(\infty)}2 (1+\cos 2\pi\theta)\right),
\end{align*}
where $h_q(x)\= \sgn (x)\cdot\tilde h_q(|x|)$ with
$\tilde h_q(x)\= \int_0^x \! (1+\zeta^q)^{-1}\, d\zeta$, which can be seen similarly as for (\ref{eq: arctan}).
Analogously, we obtain a non-smooth saddle-node bifurcation for the family
\begin{align*}
 (\theta,x)\mapsto h_q(\alpha x)-2\beta-\frac{1+\sin2\pi\theta}{2},
\end{align*}
which has been considered numerically for $q=2$ in \cite{Anagnostopoulou}, for example.

\begin{rem}
The assumption that $\alpha_c^{-1}=\alpha_e = \alpha^{\frac2{p}}$ and $\alpha_l^{-1}=\alpha_u = \alpha^p$
is only for technical reasons. It basically originates from the fact that we defined 
$\I_{1,\beta}$ in a symmetric way, that is, we considered
the intersection of the $M_0^--th$ iterate of $\I_0-(M_0^--1)\w \times C$ and
the $M_0^+$-th inverse iterate of $\I_0+(M_0^++1)\w \times E$ with $M_0^+=M_0^-$ (cf. Definition~\ref{def: critical regions}).
By allowing different relations between $M_0^+$ and $M_0^-$, we could also allow different scaling behaviour in order to apply a similar statement like
Theorem~\ref{thm: sink source orbit introduction} to
$(\theta,x)\mapsto \left(\theta+\w,\tanh(\alpha x)-\beta (1+\cos(2\pi \theta))\right)$, for example, where the ratio of $\alpha_l^{-1}/ \alpha_u$ grows exponentially with $\alpha$. 
\end{rem}

\begin{rem}
Combining Theorem~\ref{thm: concave fibre maps Keller},
Theorem~\ref{thm: anagnostopoulou}, and Theorem~\ref{thm: sink source orbit introduction}, we straightforwardly get conditions which guarantee the occurrence of a non-smooth saddle-node bifurcation. However, it is worth mentioning that besides some minor technical hypothesis, the convexity assumption of Theorem~\ref{thm: anagnostopoulou} is not needed in Theorem~\ref{thm: sink source orbit introduction}. In other words: The existence of an SNA is in a sense independent of the saddle-node bifurcation framework.
\end{rem}
\section{Proof of Theorem~\ref{thm: sink source orbit introduction}}\label{sect: proof of the main theorem}
In this section, we prove Theorem~\ref{thm: sink source orbit introduction} by showing
that there is a point $(\theta,x)$ whose positive iterates mostly stay in the expanding region, while its negative iterates mostly stay in the contracting region. This can be achieved if the frequency of the jumps from one region to the other is small enough, which is the idea behind the inductive assumptions $(\mc F1)_n$ and $(\mc F2)_n$ (Section~\ref{sect: combinatorical considerations}). These are basically hypothesis on the size of the inductively defined \emph{critical intervals} $\mc I_n$. By a geometrical argument, we will get upper bounds for these quantities in Section~\ref{sect: geometric considerations}.
In Section~\ref{sect: critical intervals}, we eventually show that these upper bounds decrease fast enough to guarantee the existence of an SNA.

\subsection{Combinatorial considerations - The basic mechanism}
\label{sect: combinatorical considerations}
We make use of $(\mc A \ref{axiom: 1})$-$(\mc A \ref{axiom: 4})$ and $(\mc A \ref{axiom: 8})$ to estimate the vertical growth rate of orbits which converge to a sink-source orbit. In order to achieve this, we need to assume some additional inductive assumptions.
The verification of these additional assumptions is the goal of the subsequent sections.
As a matter of fact, the statements of this section are basically provided in \cite{Bjerklöv, Jäger}, already. For the convenience of the reader and as there are some subtle technical differences, we nevertheless include some of the proofs.

In the following, let $\left (M_n\right)_{n\in \N_0},
\left (K_n\right)_{n\in \N_0} \in \N^{\N_0}$  be strictly increasing sequences (independent of $\beta$) with $M_0\geq 2$.
\begin{defn}\label{def: critical regions}
Suppose we have already
defined the 
$n$-th \emph{critical region} 
$\mc I_{n,\beta}$. 
Set
\begin{itemize}
 \item $\mc A_{n,\beta} \=\left (\mc I_{n,\beta} - (M_n-1)\w\right)\times C$,
 \item $\mc B_{n,\beta} \=\left(\mc I_{n,\beta} + (M_n+1)\w\right)\times E$,
 \item $\mc I_{n+1,\beta}\= 
\pi_1 \left(f_\beta^{M_n-1} (\A_{n,\beta})\cap f_\beta^{-(M_n+1)}(\B_{n,\beta})\right)$.
\end{itemize}
\end{defn}
\begin{rem}
It is obvious that $\I_{n+1,\beta}\ssq \I_{n,\beta}\ (n\in \N_0)$.
However, note that $\mc I_{n+1,\beta}$ might be empty even
if $\I_{n,\beta}\neq \emptyset$.
\end{rem}
For fixed $N\in \N$, we will only consider such $\beta \in \mathscr B(0)$ with
$f^{M_n-1}_{\beta,\theta_{-\left(M_n-1\right)}}(c^+)\geq f^{-\left(M_n+1\right)}_{\beta,\theta_{M_n+1}}(e^-)$ for each $\theta \in \I_{n}$ and $0\leq n\leq N-1$. We denote the set of these $\beta$ by $\tilde{\mathscr B}(N)$ and set
$\tilde{\mathscr B}(0)\={\mathscr B}(0)$. 

Occasionally, we might suppress the index $\beta$.
For $n\in \N_0$, set $\mc Z^-_n\= \bigcup_{j=0}^{n} \bigcup_{l=-(M_j-2)}^{0}\mc I_j+l\w$; 
$\mc Z^+_n\= \bigcup_{j=0}^{n} \bigcup_{l=1}^{M_j}\mc I_j+l\w$;
$\mc V_n  \= \bigcup_{j=0}^{n} \bigcup_{l=1}^{M_j+1}\mc I_j+l\w$;
$\mc W_n \= \bigcup_{j=0}^{n} \bigcup_{l=-(M_j-1)}^{0}\mc I_j+l\w$. Moreover, set
$\I_{-1,\beta}=\T^d$ and
$\mc Z^-_{-1},\mc Z^+_{-1},\mc V_{-1},\mc W_{-1}=\emptyset$.

In order to be able to control an orbit, we do not want it to visit the critical regions
too often. We therefore need to assume that the critical regions are small enough.
\begin{defn}
We say $f_\beta$ verifies $(\mc F1)_n$ and $(\mc F2)_n$, respectively if 
\begin{enumerate}[$(\mc F1)_n$]
 \item $\mc I_{j,\beta} \cap \bigcup_{k=1}^{2K_jM_j} \mc I_{j,\beta} +k\w= \emptyset$, \label{axiom: diophantine 1}
 \item $\left(\mc I_{j,\beta}- (M_j-1)\w \cup \mc I_{j,\beta}+ (M_j+1)\w\right) \bigcap
\left ( \mc V_{j-1} \cup \mc W_{j-1}\right)
=\emptyset$, \label{axiom: diophantine 2}
\end{enumerate}
for $j=0,\ldots,n$ and $n\in \N_0$. If $f_\beta$ satisfies both
$(\mc F1)_n$ and $(\mc F2)_n$, we say $f_\beta$ satisfies
$(\mc F)_n$. It is convenient to set $(\mc F)_{-1}$ to be true.
\end{defn}
For $\theta\in \T^d$, we denote by $\mc L_m,\mc R_m \in \N_0\cup \{\infty\}$
the smallest integers $l,r$ with $\theta_{l}\in \I_m$ and $\theta_{-r}\in \I_m+\w$, respectively.
\begin{lem}[cf. {\cite[Lemma~3.4]{Jäger}}]
\label{lem: duration of stay in contracting/expanding regions}
Let $n\in \N_0$. Suppose $f_\beta$ satisfies $(\mc A\ref{axiom: 3})$, $(\mc A\ref{axiom: 4})$ and $(\mc A\ref{axiom: 8})$ as well as $(\mc F)_{n-1}$ with
$\beta \in \tilde{\mathscr B}(n)$
and assume
\begin{align}\tag*{$(\mc B1)_n$} \label{axiom: 21}
\begin{cases}
x \in C\\
\theta \notin \mc Z^-_{n-1}.
\end{cases}
\end{align}
Furthermore, let $0<\mc L^{(1)}<\ldots <\mc L^{(N)}=\mc L_n$ be all those 
times $m\leq \mc L_n$ for which $\theta+m \w \in \I_{n-1}$. Then
$(\theta_{\mc L^{(i)}+M_{n-1}+2},x_{\mc L^{(i)}+M_{n-1}+2})$ satisfies
$(\mc B1)_n$ for each $i=1,\ldots,N-1$ and the following implication holds
\begin{align}\tag*{$(\mc C1)_n$}\label{implication: x not in C for small times..}
 x_k \notin C \Rightarrow \theta_{k} \in \mc V_{n-1} \text{ and } x_k \in [e^-,c^-]
\quad (k = 1,\ldots, \mc L_n).
\end{align}

Analogously backwards: Instead of \ref{axiom: 21}, assume
\begin{align}\tag*{$(\mc B2)_n$} \label{axiom: 22}
\begin{cases}
x \in E\\
\theta \notin  \mc Z^+_{n-1}
\end{cases}
\end{align}
and let $0<\mc R^{(1)}<\ldots <\mc R^{(N)}=\mc R_n$ be all those 
times $m\leq \mc R_n$ for which $\theta-m \w \in \I_{n-1}$.
Then $(\theta_{-\mc R^{(i)}-M_{n-1}},x_{-\mc R^{(i)}-M_{n-1}})$ satisfies
$(\mc B2)_n$ for each $i=1,\ldots,N-1$ and the following implication holds
\begin{align}\tag*{$(\mc C2)_n$}\label{implication: x not in E for small negative times..} x_{-k} \notin E \Rightarrow \theta_{-k} \in \mc W_{n-1} \text{ and } x_{-k} \in [e^+,c^+] \quad (k= 1,\ldots, \mc R_n).
\end{align}
\end{lem}
\begin{proof}
We only consider the forward case; the other case works similarly.
Note that for $n=0$ the statement
is true due to ${(\mathcal A\ref{axiom: 8})}$.

Assume the statement holds for $n_0 \in \N$ and assume $(\theta,x)$ satisfies
$(\mc B1)_{n_0+1}$. 
Trivially, $\left(\mc B1\right)_{n_0+1}$ implies $\left(\mc B1\right)_{n_0}$ such that
$x_k \notin C \Rightarrow \theta_{k} \in \mc V_{n-1}$ and $x_k \in [e^-,c^+]$
for $k\leq \mc L^{(1)}$. Notice that 
$\left(\mc I_{n_0}-(M_{n_0}-1)\w\right)\cap \mc V_{n_0-1} =\emptyset$ because of $(\mc F\ref{axiom: diophantine 2})_{n_0}$. Hence,
$\left(\theta_{\mc L^{(1)}-(M_{n_0}-1)}, x_{\mc L^{(1)}-(M_{n_0}-1)}\right) \in \mc A_{n_0}$ due to $(\mathcal C1)_{n_0}$.
As $\beta \in \tilde{\mathscr B}(n)$, we further have
$f^{M_{n_0}-1}_{\beta,\theta_{\mc L^{(1)}-M_{n_0}+1}}(c^+)\geq f^{-\left( M_{n_0}+1\right)}_{\beta,\theta_{\mc L^{(1)}+M_{n_0}+1}}(e^-)$.
If we had $x_{\mc L^{(1)}}\leq f^{-\left( M_{n_0}+1\right)}_{\beta,\theta_{\mc L^{(1)}+M_{n_0}+1}}(e^-)$, this would imply there exists $y \in [x_{\mc L^{(1)}-M_{n_0}+1},c^+]\ssq [c^-,c^+]$ with
$f^{M_{n_0}-1}_{\beta}(\theta_{\mc L^{(1)}-M_{n_0}+1},y) \in f^{-\left( M_{n_0}+1\right)}_{\beta}\left(\left\{\theta_{\mc L^{(1)}+M_{n_0}+1}\right\}\times E\right)$ meaning that $\theta_{\mc L^{(1)}} \in \mc I_{n,\beta}$, which contradicts the assumptions.
Therefore, $x_{\mc L^{(1)}}\geq f^{-\left( M_{n_0}+1\right)}_{\beta,\theta_{\mc L^{(1)}+M_{n_0}+1}}(e^-)$. 
By 
${(\mathcal A\ref{axiom: 4})}$ and the monotonicity, we thus have $x_k \in [e^-,c^+]$
for $k= \mc L^{(1)},\ldots,\mc L^{(1)}+M_{n_0}+1$. Now, $x_{\mc L^{(1)}+M_{n_0}+1} \notin E$, since otherwise again $\theta_{\mc L^{(1)}}\in \mc I_{n_0+1}$,
by definition of $\mc I_{n_0+1}$.
$\left(\mc A{\ref{axiom: 8}}\right)$,  and $(\mc F\ref{axiom: diophantine 2})_{n_0}$ hence yield $x_{\mc L^{(1)}+M_{n_0}+2} \in C$.
By $(\mc F)_{n_0}$,
we get that
$\left(\mc I_{n_0}+ (M_{n_0}+2)\w\right)\cap \mc Z_{n_0}^-=
\emptyset$.
Thus, $\left(\theta_{\mc L^{(1)}+M_{n_0}+2}, x_{\mc L^{(1)}+M_{n_0}+2}\right)$ verifies
$\left(\mc B1\right)_{n_0+1}$. 
The statement follows by induction.
\end{proof}
Notice that $(\theta,x)\in\mc A_n$ satisfies \ref{axiom: 21} with 
$\mc L_n=M_n-1$ and $(\theta,x)\in \mc B_n$ satisfies
\ref{axiom: 22} with $\mc R_n = M_n$ because of $(\mc F2)_n$.
\begin{cor}[cf. {\cite[Corollary~3.7]{Jäger}}]
\label{cor: nested sequence}
Let $N>n \in \N_0$.
Suppose $f_\beta$ satisfies
$(\mc A\ref{axiom: 3})$, $(\mc A\ref{axiom: 4})$, $(\mc A\ref{axiom: 8})$ as well as $(\mc F1)_{N-1}$, $(\mc F2)_{N}$ with $\beta \in \tilde{\mathscr B}(N)$. Then 
$f_\beta^{M_N-M_n}(\A_N)\ssq \left(\I_n-(M_n-1)\w\right) \times (c^-,c^+]\ssq \A_n$ and
$f_\beta^{-M_N+M_n}(\B_N)\ssq \left(\I_n+(M_n+1)\w\right) \times [e^-,e^+)\ssq \B_n$.
\end{cor}
\begin{proof}
Since $\mc I_{n+1}-(M_{n}-1)\w \cap \mc V_n=\emptyset$,
Lemma~\ref{lem: duration of stay in contracting/expanding regions} yields
$f_\beta^{M_{n+1}-M_n}(\A_{n+1})\ssq \A_n$. 
Due to $\left(\mc A{\ref{axiom: 8}}\right)$, the proof of Lemma~\ref{lem: duration of stay in contracting/expanding regions} even yields
the slightly stronger inclusion $f_\beta^{M_{n+1}-M_n}(\A_{n+1})\ssq \left(\I_n-(M_n-1)\w\right) \times (c^-,c^+]$. Now, the first result follows by induction.
The other relation follows similarly.
\end{proof}
By means of the next statement, we can control the time spent in the contracting
and expanding region, respectively.
For $n,N\in \N$ set
\begin{align*}
\mc P_n^N(\theta,x) &\= \# \{l \in [n,N-1]\cap \N_0 \:x_l \in C\text{ and } \theta_l\notin \I_0\}\\
\mc Q_n^N(\theta,x) &\= \# \{l \in [n,N-1] \cap \N_0 \:x_{-l} \in E\text{ and } \theta_{-l}\notin \I_0+\w\}.
\end{align*}
Notice the slight difference to the corresponding definitions in \cite{Jäger}, where 
the first coordinate had not to be taken into account.
Set
\begin{align*}
b_0\=1, \quad b_n\=\left(1-\frac1{K_{n-1}}\right)b_{n-1} \ (n\in \N).
\end{align*}

\begin{lem}[cf. {\cite[Lemma~3.8]{Jäger}}]
\label{lem: estimate for times spent in contracting/expanding regions}
Let $n\in \N_0$. Suppose $f_\beta$ satisfies
$(\mc A\ref{axiom: 3})$, $(\mc A\ref{axiom: 4})$, $(\mc A\ref{axiom: 8})$ and $(\mc F)_{n-1}$ with $\beta \in \tilde{\mathscr B}(n)$.
Furthermore, assume 
\ref{axiom: 21}
and let $0<\mc L^{(1)}<\ldots <\mc L^{(N)}=\mc L_n$ be as in Lemma~\ref{lem: duration of stay in contracting/expanding regions}.
Then, for each $i=0,\ldots,N$, we have
\begin{align*}
 P_k^{\mc L^{(i)}}(\theta,x) \geq b_n (\mc L^{(i)}-k) \quad (k=0,\ldots,\mc L^{(i)}).
\end{align*}
Analogously backwards: Instead of 
\ref{axiom: 21},
assume \ref{axiom: 22} and let $0<\mc R^{(1)}<\ldots <\mc R^{(N)}=\mc R_n$ be as in Lemma~\ref{lem: duration of stay in contracting/expanding regions}.
Then, for each $i=0,\ldots,N$, we have
\begin{align*}
 Q_k^{\mc R^{(i)}}(\theta,x) \geq b_n (\mc R^{(i)}-k) \quad (k=0,\ldots,\mc R^{(i)}).
\end{align*}
\end{lem}
\begin{rem}
For the present work (that is, in order to show the existence of an SNA)
it suffices to have the lower bound for
$P_k^{\mc L_n}(\theta,x)$ only. Nevertheless, the estimates for $P_k^{\mc L^{(i)}}(\theta,x)$ will be needed in order to study further properties of the SNA \cite{FuhrmannJäger}.
\end{rem}
\begin{proof} We consider the first inequality, the second  one follows similarly.
For $n=0$, the statement follows from $(\mc A\ref{axiom: 8})$.

Assume the statement is true for $n=n_0$ and assume $(\theta,x)$ verifies $(\mc B1)_{n_0+1}$.
Due to Lemma~\ref{lem: duration of stay in contracting/expanding regions},
we have that $(\theta_{\mc L^{(i)}+M_{n_0}+2},x_{\mc L^{(i)}+M_{n_0}+2})$
satisfies $(\mc B1)_{n_0+1}$ for $i=1,\ldots,N-1$. 
By the induction hypothesis we thus get the desired estimate for
$P_k^{\mc L^{(i)}}(\theta,x)$
as long as $k\in [\mc L^{(i)}+M_{n_0}+2,\mc L^{(i+1)}]$ for some $0\leq i \leq N-1$ or as
$i=1$ and $k\in [0,\mc L^{(1)}-1]$.

Moreover, by $(\mc F1)_{n_0}$ we have
\begin{align}\label{eq: Li-Li+1}
\mc L^{(i+1)} - \mc L^{(i)}\geq 2K_{n_0}M_{n_0}. 
\end{align}
Hence,
for all $k\in [\mc L^{(i)},\mc L^{(i)}+M_{n_0}+1]$ we get
\begin{align*}
\mc P_k^{\mc L^{(i+1)}}(\theta,x) &\geq
\mc P_{\mc L^{(i)}+M_{n_0}+2}^{\mc L^{(i+1)}}(\theta,x)\geq b_{n_0} (\mc L^{(i+1)}-(\mc L^{(i)}+M_{n_0}+2))\geq
b_{n_0} (\mc L^{(i+1)}-\mc L^{(i)}-2M_{n_0})
\\
&\stackrel{\left(\ref{eq: Li-Li+1}\right)}{\geq} b_{n_0+1} (\mc L^{(i+1)}-\mc L^{(i)})
\geq b_{n_0+1} (\mc L^{(i+1)}-k).
\end{align*}
Altogether, with $j \= \min\{l=1,\ldots,N\: \mc L^{(l)}\geq k\}$ and $N\geq i\geq j$ we therefore have
\begin{align*}
 \mc P_k^{\mc L^{(i)}}(\theta,x) = \mc P_k^{\mc L^{(j)}}(\theta,x)+ 
\sum_{l=j}^{i} \mc P_{\mc L^{(l)}}^{\mc L^{(l+1)}}(\theta,x)\geq b_{n_0+1}\left(\mc L^{(j)}-k+\sum_{l=j}^i \mc L^{(l+1)}-\mc L^{(l)}\right)=b_{n_0+1} (\mc L^{(i)}-k).
\end{align*}
\end{proof}
The following two results can be proved like the respective statements in \cite{Jäger}.
\begin{cor}[cf. {\cite[Corollary~3.9]{Jäger}}]
\label{cor: derivatives of iterates}
Let $n\in \N_0$. Suppose $f_\beta$ satisfies
$(\mc A\ref{axiom: 3})$, $(\mc A\ref{axiom: 4})$, $(\mc A\ref{axiom: 8})$ and $(\mc F)_{n-1}$ with $\beta \in \tilde{\mathscr B}(n)$.
Further, assume $(\mc A\ref{axiom: 1})$ and let $(\theta,x)\in f_\beta^{M_n-1}(\mc A_n)$. Then
\begin{align*}
 \d_x f_{\beta,\theta}^{-k}(x)\geq \left(\alpha_c^{b_n} \alpha_u^{1-b_n}\right)^{-k} \quad(0\leq k \leq M_n-1).
\end{align*}
Analogously, instead of $(\mc A\ref{axiom: 1})$ assume $(\mc A\ref{axiom: 2})$ and let $(\theta,x)\in f_\beta^{-M_n}(\mc B_n)$.
Then
\begin{align*}
 \d_x f_{\beta,\theta}^{k}(x)\geq \left(\alpha_e^{b_n} \alpha_l^{1-b_n}\right)^{k} \quad(0\leq k \leq M_n).
\end{align*}
\end{cor}
Define
\begin{itemize}
 \item $b \=\lim_{n\to \infty} b_n$,
 \item $\alpha_-\=\alpha_c^b \alpha_u^{1-b}$,
\item $\alpha_+\=\alpha_e^b \alpha_l^{1-b}$.
\end{itemize}
\begin{prop}[cf.{\cite[Proposition~3.10]{Jäger}}]
\label{prop: sink source orbit} 
Suppose $f_\beta$ satisfies
$(\mc A\ref{axiom: 1})$-$(\mc A\ref{axiom: 4})$ as well as $(\mc A\ref{axiom: 8})$ and for each
$n\in \N$ we have $f_\beta^{M_n-1}(\mc A_n)\cap f_\beta^{-M_n}(\mc B_n)\neq \emptyset$. Moreover, assume $(\mc F)_{n}$ holds for all $n\in \N$, $\alpha_-^{-1},\alpha_+>1$
and $\beta \in \bigcap_{n \in \N} \tilde{\mathscr B}(n)\neq \emptyset$.
Then there exists a sink-source orbit in $\T^d\times [e^-,c^+]$ and hence an SNA and an SNR. More precisely,
\begin{align*}
 \{(\theta,x) \in \T^d\times \X\: (\theta,x) \text{ is a sink-source orbit}\}\supseteq 
\bigcap_{n\in \N} \left (f_\beta^{M_n-1}(\mc A_n) \cap f_\beta^{-M_n}(\mc B_n) \right) 
\neq\emptyset.
\end{align*}
\end{prop}
\subsection{Geometric Considerations}\label{sect: geometric considerations}
In this paragraph, we get an upper bound for the size of the $n$-th critical region 
$\mc I_{n,\beta}$.
So far, we dealt with $\beta \in \tilde{\mathscr B}(n)$ in order to guarantee that the respective orbits stay in the strip $\T^d\times [e^-,c^+]$. Due to the monotonicity in $\beta$ (provided by ($\A\ref{axiom: 6}$)), this amounts to only considering
small enough $\beta$. 
On the other hand, $\tilde{\mathscr B}(n)$ also contains parameters $\beta$ which are \emph{too small} such that $\mc I_{n,\beta}=\emptyset$, which is not desirable either. In order to exclude these parameters as well, we define the set of \emph{admissible parameters} up to order $n\in\N$ by 
\begin{align*}
\mathscr B(n)&\=\left\{\beta \in \tilde{\mathscr B}(n)\:
f^{M_l-1}_{\beta,\theta_{-\left(M_l-1\right)}}(c^-)
\leq
f^{-\left(M_l+1\right)}_{\beta,\theta_{M_l+1}}(e^+) \text{ for some } \theta \in \I_{l} \text{ and } 0\leq l\leq n-1\right\}\\
&=\left\{\beta \in \tilde{\mathscr B}(n)\: \I_{l,\beta}\neq\emptyset \text{ for }
0\leq l \leq n
\right\}
,
\end{align*}
where we assume $\left ( M_l\right)_{l=0,\ldots,n}$ to be given.
\begin{prop}\label{prop: monotonicity}
Suppose $\left(f_\beta\right)_{\beta \in [0,1]}$ satisfies $(\A\ref{axiom: 6})$ and $(\mc A\ref{axiom: 10})$ and let $\beta<\beta' \in \tilde{\mathscr B}(n)$ for some $n\in \N_0$. Then
\begin{align}\label{eq: union of critical regions}
 \I_{n,\beta} \ssq \I_{n,\beta'}.
\end{align}
In particular, this implies that $\mathscr B(n)$ is an interval.
\end{prop}
\begin{proof}
For $n=0$, (\ref{eq: union of critical regions}) holds by
$(\mc A\ref{axiom: 10})$.
Assume (\ref{eq: union of critical regions}) is true for some $n\in \N_0$.
For $\beta \in \tilde{\mathscr B}(n+1)$, we know
$\theta \in \I_{n+1,\beta}$ if and only if
$0 \geq f^{M_{n}-1}_{\beta,\theta_{-\left(M_{n}-1\right)}}(c^-)- f^{-\left(M_{n}+1\right)}_{\beta,\theta_{M_{n}+1}}(e^+)$.
Since $f_{(\cdot)}(\theta,\x)$ is non-increasing,
$f^{M_{n}-1}_{(\cdot),\theta_{-\left(M_{n}-1\right)}}(c^-)- f^{-\left(M_{n}+1\right)}_{(\cdot),\theta_{M_{n}+1}}(e^+)$ is non-increasing, too.
Hence, $\theta \in \I_{n+1,\beta}+\w$ implies
$\theta \in \I_{n+1,\beta'}+\w$. Now, (\ref{eq: union of critical regions}) follows by induction.
\end{proof}

Up to now, we basically used monotonicity in $\beta$ in order to investigate the set of 
admissible parameters. In order to guarantee that $\mathscr B(n)$ is not empty and to control the size of the critical regions $\I_{n,\beta}$, we need 
subtler geometric information.
The intuitive idea of the argument for the smallness of $\I_{n,\beta}$ can be seen by considering $\I_{1,\beta}$:
As $f_\beta^{j}(\mc A_{0,\beta})$ stays in the contracting region for $j=0,\ldots,M_0-1$, the iterates of $\A_{0,\beta}$ become thinner and
thinner horizontal strips with each step of the iteration until they meet $\I_{0,\beta}\times C$. Likewise, 
$f_\beta^{-M_0}(\mc B_{0,\beta})$ is basically a thin horizontal strip. Iterating $f_\beta^{M_0-1}(\mc A_{0,\beta})$ once more deforms the previously horizontal strip to a thin strip around a parabola with second derivative at least $s$ because of $(\A\ref{axiom: 9})$. 
This yields an upper bound for the size of $\I_{1,\beta}$, see Figure~\ref{fig: geometric idea}\,(a).

The smallness of $\I_{n,\beta}$ follows in a similar fashion, but we have to show that even though the iterates of $\mc A_{n,\beta}$
enter the expanding region for some iterates, the overall effect of the iteration under $f$ is still a contraction.
\begin{figure}
\centering 
\subfloat[]{\includegraphics{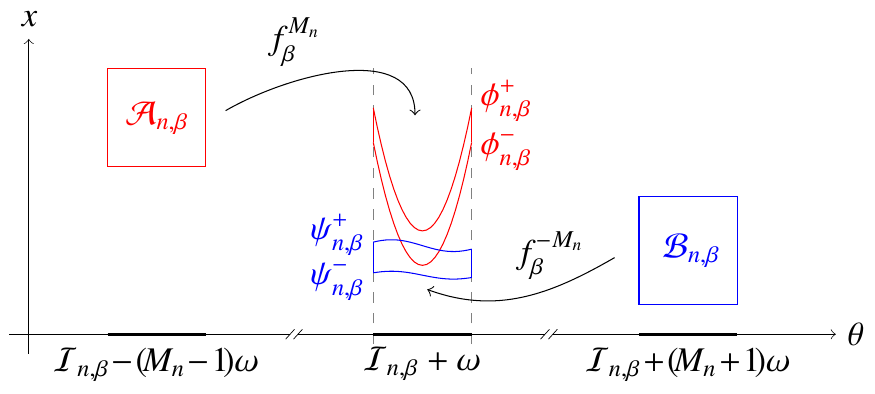}}
\subfloat[]{\includegraphics{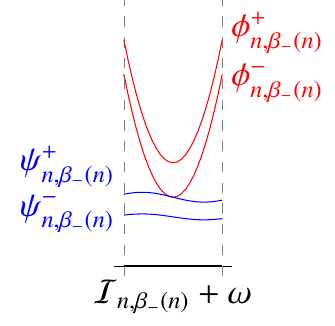}}
\subfloat[]{\includegraphics{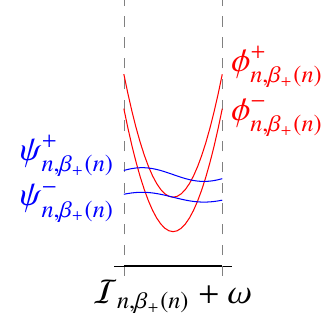}}
\caption{(a) The geometric idea behind the proof of Lemma~\ref{lem: geometric};
(b) $\I_{n+1,\beta_{-}(n+1)}$ is degenerate, with $\beta_-(n+1)=\min \mathscr B(n+1)$; (c) $\beta_+(n)=\max \mathscr B(n+1)$ is the largest parameter such that $\I_{n+1,\beta_{+}(n+1)}$ is connected.} 
\label{fig: geometric idea}
\end{figure}

In order to formalise this intuitive idea, we define the functions 
\begin{align*}
\phi_{n,\beta}^\pm(\theta)\=f^{M_{n}}_{\beta,\theta-M_{n}\w}(c^\pm) \quad
\text{and} \quad
\psi_{n,\beta}^\pm(\theta)\=f^{-M_{n}}_{\beta,\theta+M_{n}\w}(e^\pm)
\end{align*}
for $\theta \in 
\I_{n,\beta}+\w$, $n\in \N_0$. Note that
\begin{align*}
f^{M_{n}}_\beta\left(\mc A_{n,\beta}\right)&=\left\{(\theta,x)\in 
\left(\I_{n,\beta}+\w\right)\times\X\: x\in 
[\phi_{n,\beta}^-(\theta),\phi_{n,\beta}^+(\theta)]\right\},\\
f^{-M_{n}}_\beta\left(\mc B_{n,\beta}\right)&=\left\{(\theta,x)\in \left(\I_{n,\beta}+\w 
\right)\times\X\: x\in [\psi_{n,\beta}^-(\theta),\psi_{n,\beta}^+(\theta)]\right\}
\end{align*}
(cf. Figure~\ref{fig: geometric idea}).
We introduce a shorthand notation for the following inductive assumptions.
\begin{empheq}[right = \empheqbigrbrace \ (\mc I)_n]{align}
\mathscr B(n) \text{ is a non-empty and closed interval,}\label{inductive: 1}\\
\I_{n,\beta} \text{ is closed and convex
for } \beta \in \mathscr B(n)
\label{inductive: 2},
 \\ 
 \phi_{n,\beta}^-(\theta)> \psi_{n,\beta}^+(\theta) 
\text{ for each } \theta \in \d \I_{n,\beta}+\w \text{ and } \beta \in \mathscr B(n)\label{inductive: 3}, \\
\exists \beta_-(n+1)\in \B(n) \text{ and } \exists ! \theta_-^{n} \in \mc I_{n,\beta_-(n+1)}+\w \: \phi_{n,\beta_-(n+1)}^-(\theta_-^{n})=\psi_{n,\beta_-(n+1)}^+(\theta_-^{n})  \label{inductive: 4},\\
\exists \beta_+(n+1)\in \B(n) \text{ and } \exists ! \theta_+^{n} \in \mc I_{n,\beta_+(n+1)}+\w\: \phi_{n,\beta_+(n+1)}^+(\theta_+^{n})=\psi_{n,\beta_+(n+1)}^-(\theta_+^{n}). \label{inductive: 5}
\end{empheq}
Moreover, set
\begin{itemize}
\item $H_n^\phi\=\sup\limits_{\theta \in \mc I_n, \beta \in \mathscr{B}(n)} |\phi_{n,\beta}^+(\theta)-\phi_{n,\beta}^-(\theta)|$,
\item $H_n^\psi\=\sup\limits_{\theta \in \mc I_n, \beta \in \mathscr{B}(n)} |\psi_{n,\beta}^+(\theta)-\psi_{n,\beta}^-(\theta)|$,
\item $v_n^\tau \=\inf\limits_{\substack{\theta \in \mathring \I_{n},\, \beta \in \mathscr{B}(n)
\\ \vartheta \in \mathbb S^{d-1}}} \d_\vartheta^2 \phi^\tau_{n,\beta}(\theta)-\d_\vartheta^2 \psi^{-\tau}_{n,\beta}(\theta)$ $(\tau \in \{-,+\})$.
\end{itemize}
\begin{lem}\label{lem: geometric}
Assume $(\I)_n$ holds for some $n\in \N_0$.
Then $\mathscr B(n+1)$ is non-empty. 
Further, suppose $f_\beta$ satisfies $(\A\ref{axiom: 3})$, $(\A\ref{axiom: 4})$, $(\mc A\ref{axiom: 6})$-$(\mc A\ref{axiom: 8})$, $(\mc A \ref{axiom: 10})$ and $(\mc F1)_{n}$, $(\mc F2)_{n+1}$
for $\beta \in \mathscr B(n+1)$.
If $ v_{n}^\pm,  v_{n+1}^\pm>0$, then
\begin{itemize}
 \item $\left(\I\right)_{n+1}$ holds,
 \item $\left| \mc I_{n+1,\beta} \right|\leq \sqrt8\sqrt{ 
\frac{H_{n}^\phi+H_{n}^\psi}{v_{n}^-}}$ for $\beta \in \mathscr B(n+1)$.
\end{itemize}
\end{lem}
\begin{proof}
Note that $\emptyset \neq \mathscr B(n+1) =[\beta_-(n+1),\beta_+(n+1)]$ by (\ref{inductive: 4}), (\ref{inductive: 5}) as well as Proposition~\ref{prop: monotonicity} and ($\A\ref{axiom: 6}$).

$\I_{n+1,\beta}$ is a sublevel set of
$\phi_{n,\beta}^--\psi^+_{n,\beta}$ and hence, it is closed.
Given two points $\theta_1,\theta_2 \in \I_{n+1,\beta}$, denote by $[\theta_1,\theta_2]\ssq\I_{n,\beta}$
the line joining the two points.
As $\d_\vartheta^2 \phi^-_{n,\beta}(\theta)-\d_\vartheta^2 \psi^{+}_{n,\beta}(\theta)\geq\nu_{n}^->0$ (with $\vartheta$ the unit vector in direction of $\theta_2-\theta_1$), we have $\phi_{n,\beta}^--\psi^+_{n,\beta}\leq 0$ on $[\theta_1,\theta_2]$ and thus convexity of $\I_{n+1,\beta}$.

By Corollary~\ref{cor: nested sequence}, $[\phi^-_{n+1,\beta}(\theta),\phi^+_{n+1,\beta}(\theta)]\ssq(\phi^-_{n,\beta}(\theta),\phi^+_{n,\beta}(\theta)]$ and $[\psi^-_{n+1,\beta}(\theta),\psi^+_{n+1,\beta}(\theta)]\ssq[\psi^-_{n,\beta}(\theta),\psi^+_{n,\beta}(\theta))$
for all $\theta \in \I_{n+1,\beta}+\w$, $\beta \in \mathscr{B}(n)$.
This ensures 
 $\phi_{n+1,\beta}^->\psi_{n+1,\beta}^+$ on $\d \I_{n+1,\beta}$ and
guarantees that
\begin{align*}
\beta_+(n+2)&\=\min\left\{\beta \in \mathscr{B}(n+1)\left| \ \exists \theta\in 
\I_{n+1,\beta}+\w \: \phi_{n+1,\beta}^+(\theta)\leq\psi_{n+1,\beta}^-(\theta)\right. \right\}
\\
&=\min\left\{\beta \in \mathscr{B}(n+1)\left| \ \exists \theta\in \I_{n+1,\beta}+\w \: \phi_{n+1,\beta}^+(\theta)=\psi_{n+1,\beta}^-(\theta)\right. \right\}
\intertext{as well as}
\beta_-(n+2)&\=\max \left\{\beta \in \operatorname{int}\mathscr{B}(n+1)\left|\ \beta< \beta_+(n+2),\ \forall \theta\in \I_{n+1,\beta}+\w \: \phi_{n+1,\beta}^-(\theta)\geq \psi_{n+1,\beta}^+(\theta)\right. \right\}
 \end{align*}
are well-defined.
Using $\nu_{n+1}^\pm >0$, we get the uniqueness of the tangent points
of $\phi_{n+1,\beta_-(n+1)}^-$ and $\psi_{n+1,\beta_-(n+1)}^+$ as well as of
$\phi_{n+1,\beta_+(n+1)}^+$ and $\psi_{n+1,\beta_+(n+1)}^-$ and conclude $\left(\I\right)_{n+1}$.

An upper bound for the size of $\I_{n+1,\beta}$ is given by the distance of the zeros of the function $\theta\mapsto \frac{\nu_{n}^-}2 \theta^2- \left(H_{n}^\phi+H_{n}^\psi\right )$. 
\end{proof}

\begin{rem}\label{rem: I0}
 By means of $(\mc A\ref{axiom: 5})$, we defined $\mathscr B(0)\ssq [0,1]$ to be a 
closed interval.
Further, we set
$\mc I_{0,\beta}\ssq \T^d$ to be closed and convex for each $\beta \in \mathscr B(0)$. 
Moreover, with $(\mc A\ref{axiom: 1})$ we have (\ref{inductive: 3}) and
similarly as in the proof of Lemma~\ref{lem: geometric} we can define $\beta_-(1)$ and $\beta_+(1)$ such that
(\ref{inductive: 4}) and (\ref{inductive: 5}) are verified, respectively, assuming
that $(\mc A\ref{axiom: 9})$ holds. In other words: $(\I)_{0}$ is true.
\end{rem}

We next provide estimates for the quantities used in Lemma~\ref{lem: geometric}.
\begin{lem}[cf. {\cite[Lemma~3.13]{Jäger}}]\label{lem: height of strips}
Let $n\in \N_0$, $\beta \in \tilde{\mathscr B}(n)$. Suppose $f_\beta$ verifies $(\I)_n$ as well as $(\mc A{\ref{axiom: 1}})-(\mc A{\ref{axiom: 4}})$ and $(\mc A{\ref{axiom: 8}})$.
If $(\mc F)_{n-1}$ holds, then $H_{n,\beta}^\phi \leq \left(\alpha_c^{b_n}\alpha_u^{1-b_n}\right)^{M_n}\cdot|C|$ and 
$H_{n,\beta}^\psi \leq \left(\alpha_e^{b_n}\alpha_l^{1-b_n}\right)^{-M_n}\cdot|E|$.
\end{lem}
\begin{proof}
Apply Corollary~\ref{cor: derivatives of iterates}.
\end{proof}
The following statement is, from a technical point of view, the core part of this work.
It provides us with a positive lower bound for $\nu_n^{\pm}$ and thereby ensures that we can apply Lemma~\ref{lem: geometric}. The idea is to show that the second derivative of $\phi^\pm_n(\theta)-\psi^\mp_n(\theta)=f_{\theta-M_n\w}^{M_n}(c^\pm)-
f_{\theta+M_n\w}^{-M_n}(c^\mp)$ in direction $\vartheta$
only differs from
$\left(\d_\vartheta^2f_{\theta-\w}\right)\left(f_{\theta-M_n\w}^{M_n-1}(c^\pm)\right)$ by a
remainder term, whose supremum goes to zero exponentially fast with increasing $\alpha$.
Since $(\mc A\ref{axiom: 9})$ provides us with a lower bound $s$ for the second derivative
of $f$ with respect to the base coordinates in every direction, this proves the claim.
\begin{lem}\label{lem: nu n} Let $n\in \N_0$, $\beta \in \tilde{\mathscr B}(n)$.
Suppose $f_\beta$ satisfies $(\A\ref{axiom: 1})$-$(\A\ref{axiom: 4})$, $(\A\ref{axiom: 8})$, $(\A\ref{axiom: 9})$, $(\A\ref{axiom: 11})$-$(\A\ref{axiom: 16})$ and $(\mc F)_{n-1}$. Let there be $p\geq \sqrt{2}$ and $\alpha>1$ such that
\begin{align*}
 \alpha_c^{-1}=\alpha_e = \alpha^{\frac2{p}}, \qquad \alpha_l^{-1}=\alpha_u = \alpha^p
\end{align*} 
and assume $b_n>\frac{5p^2}{2+5p^2}$. Then  
\begin{align*}
\nu^{\pm}_{n} \geq
s- S^2 c \cdot \alpha^{-\left(\frac{2b_n}p-5(1-b_n)p\right)},
\end{align*}
where $c=c(\alpha,b_n)>0$ can be chosen to be monotonously decreasing in $\alpha$ and $b_n$.
\end{lem}
\begin{proof}
For reasons of readability, we omit the index $\beta$ in the following.
Let us consider $\frac{\d^2}{\d\vartheta^2} \phi_{n}^\pm(\theta)$ ($\theta \in \mc I_n+\w$ and $\vartheta \in \mathbb S^{d-1}$). Set 
$\theta_0 \= \theta-M_n\w$.
Then 
\begin{align}\label{eq: d phi}
\begin{split}
\frac{\d}{\d\vartheta} \phi_{n}^\pm (\theta)&=\d_\vartheta f_{\theta_0}^{M_n}(c^\pm)=
\left(\d_\vartheta f_{\theta_{M_n-1}}\right)(x_{M_n-1})+\left(\d_x f_{\theta_{M_n-1}}\right)(x_{M_n-1})\cdot \d_\vartheta f_{\theta_0}^{M_n-1}(c^\pm)=\ldots= \\
&=\left(\d_\vartheta f_{\theta_{M_n-1}}\right)(x_{M_n-1})+\sum_{k=0}^{M_n-2} \left(\d_\vartheta f_{\theta_k}\right)(x_k) \cdot \left(\d_x f_{\theta_{k+1}}^{M_{n}-k-1}\right)(x_{k+1}),
\end{split}
\end{align}
where we used
\begin{align}\label{eq: d2 f}
\left(\d_x f_{\theta_{k+1}}^{M_{n}-k-1}\right)(x_{k+1})=\prod_{j=k+1}^{M_n-1} \left(\d_x f_{\theta_j}\right)(x_j) \quad (k=-1,0,\ldots,M_n-1).
\end{align}
Differentiating once more gives
\begin{align*}
\frac{\d^2}{\d\vartheta^2} \phi^\pm_n(\theta)=& \ \left(\d_\vartheta^2 f_{\theta_{M_n-1}}\right)(x_{M_n-1})\!+\!
\left(\d_x \d_\vartheta f_{\theta_{M_n-1}}\right)(x_{M_n-1})\!\cdot\!  \d_\vartheta f^{M_n-1}_{\theta_0}(c^\pm)\\
& +\sum_{k=0}^{M_n-2} \left(\d_\vartheta f_{\theta_k}\right)(x_k) \cdot \d_\vartheta \left (\d_x f_{\theta_{k+1}}^{M_{n}-k-1}\right)(x_{k+1})+ \left [\d_\vartheta \left(\d_\vartheta f_{\theta_k}\right)(x_k)\right] \cdot \left(\d_x f_{\theta_{k+1}}^{M_{n}-k-1}\right)(x_{k+1}).
\end{align*}
Further,
\begin{align*}
 \d_\vartheta  \left(\d_x f_{\theta_{k+1}}^{M_{n}-k-1}\right)(x_{k+1})&=\d_\vartheta \prod_{j=k+1}^{M_n-1} \left(\d_x f_{\theta_j}\right)(x_j)= \sum_{l=k+1}^{M_n-1} \d_\vartheta \left(\d_x f_{\theta_l}\right)(x_l) \prod_{\substack{j=k+1\\j \neq l}}^{M_n-1} \left(\d_x f_{\theta_j}\right)(x_j) \\
&=\sum_{l=k+1}^{M_n-1} \left[\left(\d_\vartheta\d_x f_{\theta_l}\right)(x_l)+
\left(\d_x^2 f_{\theta_l}\right)(x_l) \cdot \d_\vartheta f_{\theta_0}^l(c^\pm)\right] \prod_{\substack{j=k+1\\j \neq l}}^{M_n-1} \left(\d_x f_{\theta_j}\right)(x_j)
\end{align*}
and
\begin{align*}
 \d_\vartheta \left(\d_\vartheta f_{\theta_k}\right)(x_k)=\left(\d_\vartheta^2 f_{\theta_k}\right)(x_k)+
\left(\d_x \d_\vartheta f_{\theta_k}\right)(x_k) \cdot \d_\vartheta f^k_{\theta_0}(c^\pm).
\end{align*}
Altogether, we have
\begin{align}\label{eq: d 2 phi}
\begin{split}
 \frac{\d^2}{\d\vartheta^2} \phi_n^\pm(\theta)
=& \left(\d_\vartheta^2 f_{\theta_{M_n-1}}\right)(x_{M_n-1})\!+\!
\left(\d_x \d_\vartheta f_{\theta_{M_n-1}}\right)(x_{M_n-1})\!\cdot\!  \d_\vartheta f^{M_n-1}_{\theta_0}(c^\pm)\\
&+\sum_{k=0}^{M_n-2} \left(\d_\vartheta f_{\theta_k}\right)(x_k)  \left(\sum_{l=k+1}^{M_n-1} \left[\!\left(\d_\vartheta\d_x f_{\theta_l}\right)(x_l)\!+\!
\left(\d_x^2 f_{\theta_l}\right)(x_l)\! \cdot  \! \d_\vartheta f_{\theta_0}^l(c^\pm)\right] \prod_{\substack{j=k+1\\j \neq l}}^{M_n-1} \left(\d_x f_{\theta_j}\right)(x_j)\right)\\
& + \left[\left(\d_\vartheta^2 f_{\theta_k}\right)(x_k)\!+\!
\left(\d_x \d_\vartheta f_{\theta_k}\right)(x_k)\!\cdot\!  \d_\vartheta f^k_{\theta_0}(c^\pm)\right] \left(\d_x f_{\theta_{k+1}}^{M_{n}-k-1}\right)(x_{k+1}).
\end{split}
\end{align}
It is our goal to show that the long times spent in the contracting region
keep the derivatives small, such that
$\left(\d_\vartheta^2 f_{\theta_{M_n-1}}\right)(x_{M_n-1})$ becomes the
leading term.
The part which is the hardest to control is
\begin{align*}
\sum_{k=0}^{M_n-2} \left(\d_\vartheta f_{\theta_k}\right)(x_k)  \sum_{l=k+1}^{M_n-1} \left(\d_x^2 f_{\theta_l}\right)(x_l)\! \cdot  \! \d_\vartheta f_{\theta_0}^l(c^\pm) \prod\limits_{\substack{j=k+1\\j \neq l}}^{M_n-1} \left(\d_x f_{\theta_j}\right)(x_j),
\end{align*}
with $\d_\vartheta f_{\theta_0}^l(c^\pm)=\sum_{m=0}^{l-1} \left(\d_\vartheta f_{\theta_m}\right)(x_m) \cdot \left(\d_x f_{\theta_{m+1}}^{l-m-1}\right)(x_{m+1})$ as in equation~(\ref{eq: d phi}). Using $(\mc A\ref{axiom: 11})$, we see that it is bounded from above by
\begin{align}\label{eq: intermediate lower bound on d 2 phi}
S^2\sum_{k=0}^{M_n-2}\sum_{l=k+1}^{M_n-1}\sum_{m=0}^{l-1}\left|\left(\d_x^2 f_{\theta_l}\right)(x_l)\right|\left(\d_x f_{\theta_{m+1}}^{l-m-1}\right)(x_{m+1})\prod_{\substack{j=k+1\\j \neq l}}^{M_n-1} \left(\d_x f_{\theta_j}\right)(x_j).
\end{align}
If $m\leq k$, then
\begin{align*}
& 
\left|\left(\d_x^2 f_{\theta_l}\right)(x_l)\right| 
\left(\d_x f_{\theta_{m+1}}^{l-m-1}\right)(x_{m+1})\prod\limits_{\substack{j=k+1\\j \neq l}}^{M_n-1} \left(\d_x f_{\theta_j}\right)(x_j)
=
\left|\left(\d_x^2 f_{\theta_l}\right)(x_l)\right| 
\prod\limits_{\substack{j=m+1\\j \neq l}}^{M_n-1} \left(\d_x f_{\theta_j}\right)(x_j)\cdot
\prod\limits_{j=k+1}^{l-1} \left(\d_x f_{\theta_j}\right)(x_j)
\\
&\leq 
\left|\left(\d_x^2 f_{\theta_l}\right)(x_l)\right|
\prod\limits_{\substack{j=m+1\\j \neq l\\ x_j \in C }}^{M_n-1} \alpha_c \cdot
\prod\limits_{\substack{j=m+1\\j \neq l\\ x_j \notin C }}^{M_n-1} \alpha_u
\prod\limits_{\substack{j=k+1\\ x_j \notin C}}^{l-1} \alpha_u
\leq 
\left|\left(\d_x^2 f_{\theta_l}\right)(x_l)\right| 
\prod\limits_{\substack{j=m+1\\j \neq l\\ x_j \in C }}^{M_n-1} \alpha_c \cdot
\prod\limits_{\substack{j=m+1\\j \neq l\\ x_j \notin C }}^{M_n-1} \alpha_u^2
\leq \alpha_c^{b_n (M_n-m-1)} \alpha_u^{2 (1-b_n)(M_n-m-1)} \\
& \stackrel{(\A\ref{axiom: 14})}{=}\alpha_1^{-(M_n-m-1)},
\end{align*}
where we used Lemma~\ref{lem: estimate for times spent in contracting/expanding regions} in the last estimate and where we set $\alpha_1\=\alpha_c^{-b_n}\alpha_u^{-2(1-b_n)}=\alpha^{-2\left(p(1-b_n)-\frac{b_n}p\right)}$. 
For $m>k$ we get an analogous result with $m$ replaced by $k$.
Hence, (\ref{eq: intermediate lower bound on d 2 phi}) is bounded by
\begin{align*}
& S^2\sum_{k=0}^{M_n-2}\sum_{l=k+1}^{M_n-1}
\left(\sum_{m=0}^{k} \alpha_1^{-(M_n-1-m)}+\sum_{m=k+1}^{l-1} \alpha_1^{-(M_n-1-k)} \right)\\
& \leq S^2\sum_{k=0}^{M_n-2}\sum_{l=k+1}^{M_n-1}
\left(\alpha_1^{-(M_n-1-k)} \sum_{m=0}^{k} \alpha_1^{-m}+ \alpha_1^{-(M_n-1-k)} \sum_{m=k+1}^{l-1} 1  \right)
\\
& \leq S^2\sum_{k=0}^{M_n-2}
\left(\alpha_1^{-(M_n-1-k)} \frac1{1-\alpha_1^{-1}} (M_n-k-1)+ \alpha_1^{-(M_n-1-k)} \sum_{l=k+1}^{M_n-1} (l-k-1)  \right)\\
&\leq S^2 \frac2{1-\alpha_1^{-1}} \sum_{k=0}^{M_n-2}
\alpha_1^{-(M_n-1-k)} (M_n-k-1)^2\leq 
S^2 \frac{2\alpha_1}{\alpha_1-1} \sum_{l=1}^{M_n-1} l^2 \alpha_1^{-l}\\
& \leq S^2 \tilde c(\alpha_1) \cdot \alpha_1^{-1},
\end{align*}
where $\tilde c(\alpha)\=\frac{2\alpha}{\alpha-1}\sum_{l=1}^{\infty} l^2 \alpha^{-l+1}$ for each $\alpha>1$.\footnote{Notice that $\alpha_1>1$, since $b_n>\frac{5p^2}{2+5p^2}>\frac{p^2}{p^2+1}$.} Note that $\tilde c$ is monotonously decreasing in $\alpha$. 
The other addends of (\ref{eq: d 2 phi}) can be treated
in a similar fashion, which eventually gives
\begin{align*}
 \frac{\d^2}{\d\vartheta^2} \phi_n^\pm(\theta)\geq
\left(\d_\vartheta^2 f_{\theta_{M_n-1}}\right)(x_{M_n-1})-5S^2\tilde c(\alpha_1)\cdot \alpha_1^{-1}\stackrel{(\mc A\ref{axiom: 9})}{\geq} s -5S^2\tilde c(\alpha_1)\cdot \alpha_1^{-1}.
\end{align*}

Now, let us consider
$\frac{\d^2}{\d\vartheta^2}\psi_n^\pm(\theta)= f^{-M_n}_{\theta+ M_n \w} (e^\pm)$ for $\vartheta \in \mathbb S^{d-1}$. We proceed similarly as before
but this time considering the map $f^{-1}$ instead of $f$, that is $\theta_k=\theta_0-k\w$ (with $\theta_0=\theta+M_n \w$) and $x_k=f_{\theta_0}^{-k}(e^\pm)$.
\begin{align}\label{eq: derivative inverse}
\begin{split}
\d_x f_{\theta}^{-1}(x)&=\frac1{(\d_x f_{\theta-\w})\left(f^{-1}_\theta(x)\right)}
\quad \left(\Rightarrow 0<\d_x f_{\theta}^{-1}(x)<\alpha_e^{-1} \ \left(x\in E, \theta\notin \I_0+\w \right) \right),\\
\d_\vartheta f_{\theta}^{-1}(x)&=-\frac{(\d_\vartheta f_{\theta-\w})\left(f^{-1}_\theta(x)\right)}{(\d_x f_{\theta-\w})\left(f^{-1}_\theta(x)\right)}=
-(\d_\vartheta f_{\theta-\w})\left(f^{-1}_\theta(x)\right)\cdot \d_x f_{\theta}^{-1}(x),
\end{split}
\end{align}
for $(\theta,x)\in f(\T^d\times \X)\setminus\left\{(\theta,x)\:Df^{-1} \text{is singular}\right\}\supseteq f(\T^d\times[e^-,c^+])$.
Hence,
\begin{align}
\label{eq: d2 d2 f inverse}
\d_x^2 f_{\theta}^{-1}(x)=&\ -\frac{(\d_x^2 f_{\theta-\w})\left(f^{-1}_\theta(x)\right)\cdot \d_x f^{-1}_\theta(x)}{\left[(\d_x f_{\theta-\w})\left(f^{-1}_\theta(x)\right)\right]^2}=
-(\d_x^2 f_{\theta-\w})\left(f^{-1}_\theta(x)\right)\cdot \left(\d_x f_{\theta}^{-1}(x)\right)^3
\intertext{such that $\left |\d_x^2 f_{\theta}^{-1}(x)\right|\leq \alpha_u^2 \alpha_l^{-3}$ for $x\in f_{\theta-\w}([e^-,c^+])$,}
\begin{split}\label{eq: d1 d2 f inverse} 
\d_\vartheta\d_x f_{\theta}^{-1}(x)=&\ -\frac{(\d_\vartheta\d_x f_{\theta-\w})\left(f^{-1}_\theta(x)\right)+(\d_x^2 f_{\theta-\w})\left(f^{-1}_\theta(x)\right) \cdot\d_\vartheta f^{-1}_\theta(x)}{\left[(\d_x f_{\theta-\w})\left(f^{-1}_\theta(x)\right)\right]^2}
\\
=&\ -(\d_\vartheta\d_x f_{\theta-\w})\left(f^{-1}_\theta(x)\right) \cdot\left(\d_x f_{\theta}^{-1}(x)\right)^2-(\d_\vartheta f_{\theta-\w})\left(f^{-1}_\theta(x)\right)\cdot
\d_x^2 f_{\theta}^{-1}(x)
\end{split}
\end{align}
and thus 
$\left|\d_\vartheta\d_x f_{\theta}^{-1}(x)\right|\leq 2 S \alpha_u^2 \alpha_l^{-3}$
for $x\in f_{\theta-\w}([e^-,c^+])$. Finally,
\begin{align}
\label{eq: d1 d1 f inverse}
\begin{split}
\d_\vartheta^2 f_{\theta}^{-1}(x)=&
-(\d_\vartheta^2 f_{\theta-\w})\left(f^{-1}_\theta(x)\right)\cdot\d_x f_{\theta}^{-1}(x)
-(\d_x\d_\vartheta f_{\theta-\w})\left(f^{-1}_\theta(x)\right)\cdot\d_\vartheta f_{\theta}^{-1}(x) \d_x f_{\theta}^{-1}(x)\\
&\qquad\qquad\qquad\qquad\qquad\qquad\quad\qquad \ \ -(\d_\vartheta f_{\theta-\w})\left(f^{-1}_\theta(x)\right)\cdot\d_\vartheta\d_x f_{\theta}^{-1}(x)\\
=&-(\d_\vartheta^2 f_{\theta-\w})\left(f^{-1}_\theta(x)\right)\cdot\d_x f_{\theta}^{-1}(x)
-2(\d_\vartheta f_{\theta-\w})\left(f_\theta^{-1}(x)\right)\cdot \d_\vartheta\d_xf_\theta^{-1}(x)
-\left((\d_\vartheta f_{\theta-\w})\left(f^{-1}_\theta(x)\right)\right)^2\cdot
\d_x^2 f_\theta^{-1}(x).
\end{split}
\end{align}
As in the forward case, we get
\begin{align}
\label{eq: d2 psi}
\begin{split}
 \frac{\d^2}{\d\vartheta^2} \psi_n^\pm(\theta)
=& \sum_{k=0}^{M_n-1} \left(\d_\vartheta f^{-1}_{\theta_k}\right)(x_k)  \left(\sum_{l=k+1}^{M_n-1} \left[\!\left(\d_\vartheta\d_x f^{-1}_{\theta_l}\right)(x_l)\!+\!
\left(\d_x^2 f^{-1}_{\theta_l}\right)(x_l)\! \cdot  \! \d_\vartheta f^{-l}_{\theta_0}(e^\pm)\right] \prod_{\substack{j=k+1\\j \neq l}}^{M_n-1} \left(\d_x f^{-1}_{\theta_j}\right)(x_j)\right)\\
& \qquad \qquad \qquad \ \ + \left[\left(\d_\vartheta^2 f^{-1}_{\theta_k}\right)(x_k)\!+\!
\left(\d_\vartheta \d_x f^{-1}_{\theta_k}\right)(x_k)\!\cdot\!  \d_\vartheta f^{-k}_{\theta_0}(e^\pm)\right] \left(\d_x f_{\theta_{k+1}}^{-(M_{n}-k-1)}\right)(x_{k+1}).
\end{split}
\end{align}
Similarly as before, we want to show that the long times spent in the expanding region
keep all the derivatives small (as we consider iterates of the inverse map).
Since $\d_\vartheta f_{\theta_0}^{-l}(e^\pm)=\sum_{m=0}^{l-1} \left(\d_\vartheta f_{\theta_m}^{-1}\right)(x_m) \cdot \left(\d_x f_{\theta_{m+1}}^{-(l-m-1)}\right)(x_{m+1})$, the term which is the hardest to control in (\ref{eq: d2 psi}) is
\begin{align*}
\begin{split}
&\sum_{k=0}^{M_n-1} \left(\d_\vartheta f^{-1}_{\theta_k}\right)(x_k)
\sum_{l=k+1}^{M_n-1}\left(\d_x^2 f^{-1}_{\theta_l}\right)(x_l) \d_\vartheta f^{-l}_{\theta_0}(e^\pm)
\prod_{\substack{j=k+1\\j \neq l}}^{M_n-1} \left(\d_x f^{-1}_{\theta_j}\right)(x_j)\\
&= -\sum_{k=0}^{M_n-2} (\d_\vartheta f_{\theta_{k+1}})\left(x_{k+1}\right)\d_x f_{\theta_k}^{-1}(x_k)
 \sum_{l=k+1}^{M_n-1}
  (\d_x^2 f^{-1}_{\theta_{l}})(x_{l}) 
\sum_{m=0}^{l-1} 
(\d_\vartheta f_{\theta_{m+1}})\left(x_{m+1}\right)\d_x f_{\theta_m}^{-1}(x) 
\left(\d_x f_{\theta_{m+1}}^{-(l-m-1)}\right)(x_{m+1})\\
&\qquad \qquad\qquad\qquad\qquad\qquad\qquad\qquad\qquad\qquad\qquad\qquad\qquad
\qquad\qquad\quad \times \prod_{\substack{j=k+1\\j \neq l}}^{M_n-1} \left(\d_x f^{-1}_{\theta_j}\right)(x_j).
\end{split}
\end{align*}
An upper bound for this expression
reads
\begin{align}\label{eq: backward hardest term}
 &S^2 \sum_{k=0}^{M_n-2}\sum_{l=k+1}^{M_n-1}\sum_{m=0}^{l-1}
\left |(\d_x^2 f^{-1}_{\theta_{l}})(x_{l})\right|\cdot
\prod_{\substack{n=m}}^{l-1} \left(\d_x f^{-1}_{\theta_n}\right)(x_n)\prod_{\substack{j=k\\j\neq l}}^{M_n-1} \left(\d_x f^{-1}_{\theta_j}\right)(x_j).
\end{align}
We deal similarly with (\ref{eq: backward hardest term}) as we did with
(\ref{eq: intermediate lower bound on d 2 phi}).
Suppose $ m>k$. 
Since
$\left(\d_xf^{-1}_{\theta}(x)\right)^2<\alpha_l^{-2}<\alpha_u^2\alpha_l^{-3}$, we get 
\begin{align*}
\left|(\d_x^2 f^{-1}_{\theta_{l}})(x_{l})\right|\prod\limits_{n=m}^{l-1} \left(\d_x 
f^{-1}_{\theta_n}\right)(x_n)\prod\limits_{\substack{j=k\\j\neq l}}^{M_n-1} 
\left(\d_xf^{-1}_{\theta_j}\right)(x_j)
&\leq \left|(\d_x^2 f^{-1}_{\theta_{l}})(x_{l})\right|
\prod\limits_{\substack{j=k \wedge j\neq l\\ x_j\in E\wedge \theta_j \notin \I_0+\w}}^{M_n-1}
\left(\d_xf^{-1}_{\theta_j}\right)(x_j)
\prod\limits_{\substack{j=k \wedge j\neq l\\ x_j\notin E\vee \theta_j \in \I_0+\w}}^{M_n-1} \left(\d_x 
f^{-1}_{\theta_n}\right)^2(x_n)
\\
&\stackrel{(\mc A\ref{axiom: 15})}{\leq} \alpha_e^{-b_n(M_n-k)}\left(\alpha_u^2\alpha_l^{-3}\right)^{(1-b_n)(M_n-k)}\leq \alpha_2^{-(M_n-k)}, 
\end{align*}
where $\alpha_2\=
\alpha_e^{b_n}\left(\alpha_u^2\alpha_l^{-3}\right)^{-(1-b_n)}=\alpha^{\frac{2b_n}p-5(1-b_n)p}$.
For $m\leq k$ we get an analogous result. Hence, (\ref{eq: backward hardest term}) is bounded by
\begin{align*}
&S^2\sum_{k=0}^{M_n-2}\sum_{l=k+1}^{M_n-1}
\left(\alpha_2^{-(M_n-k)} \sum_{m=0}^{k} \alpha_2^{-m} +\alpha_2^{-(M_n-k)} \sum_{m=k+1}^{l-1}\right)
\\
&\leq S^2 \frac{2\alpha_2}{\alpha_2-1} \sum_{k=0}^{M_n-2}
\alpha_2^{-(M_n-k)} (M_n-k)^2 
\leq S^2 \tilde c(\alpha_2) \cdot \alpha_2^{-2},
\end{align*}
with $\tilde c$ as in the forward case.\footnote{$\alpha_2>1$, since $b>1-\frac2{2+5p^2}$.}
Nevertheless, notice that
\begin{align*}
& \left|\sum_{k=0}^{M_n-1} \left(\d_\vartheta^2 f^{-1}_{\theta_k}\right)(x_k)\left(\d_x f_{\theta_{k+1}}^{-(M_{n}-k-1)}\right)(x_{k+1})\right|
\\
&
\leq\sum_{k=0}^{M_n-1} S^2 \left(\d_x f_{\theta_{k}}^{-(M_{n}-k)}\right)(x_{k})
+2S \left|\d_\vartheta\d_xf_{\theta_k}^{-1}(x_k)\right|\left(\d_x f_{\theta_{k+1}}^{-(M_{n}-k-1)}\right)(x_{k+1})
+S^2 \d_x^2 f_{\theta_k}^{-1}(x_k)\left(\d_x f_{\theta_{k+1}}^{-(M_{n}-k-1)}\right)(x_{k+1})
\\
&
\leq 3S^2 \tilde c(\alpha_2) \alpha_2^{-1},
\end{align*}
where we used (\ref{eq: d1 d1 f inverse}) in the first step.
Altogether, we eventually get
\begin{align*}
\frac{\d^2}{\d\vartheta^2} \psi^\pm(\theta) \leq 6 S^2 \tilde c(\alpha_2)\cdot \alpha_2^{-1}.
\end{align*}
Setting $c(\alpha,b_n)\= 6 \tilde c\left(\alpha^{\frac{2b_n}p-5(1-b_n)p}\right)+5\tilde c\left(\alpha^{\frac{2b_n}p-2(1-b_n)p}\right)$
yields the desired estimate.
\end{proof}
\subsection{Existence of a sink-source orbit}\label{sect: critical intervals}
In Section~\ref{sect: combinatorical considerations}, we proved the existence of a
sink source orbit for $f_\beta$ provided there are strictly increasing sequences $\left(M_n\right)_{n\in \N_0},\left(K_n\right)_{n\in \N_0} \in \N^{\N_0}$ such that the inductively defined
critical regions $\I_{n,\beta}$ are non-empty and satisfy $(\mc F \ref{axiom: diophantine 1})_n, (\mc F \ref{axiom: diophantine 2})_n$.
By means of the geometric considerations of the last section, we are now able to
show that for some $\beta$ such sequences $\left(M_n\right)_{n\in \N_0},\left(K_n\right)_{n\in \N_0}$ actually do exist. This finishes the proof of Theorem~\ref{thm: sink source orbit introduction}.

As a matter of fact, we are going to prove that for some $\beta$, the critical regions 
satisfy a slightly stronger version of
$(\mc F \ref{axiom: diophantine 1})_n$, that is, we will show
\begin{align}\label{eq: defn diophantine prime}
\tag*{$(\mc F1)_n'$}
d\left(\mc I_{j,\beta} \, , \bigcup_{k=1}^{2K_jM_j} \mc I_{j,\beta} +k\w\right)>|\mc I_{j,\beta}|
\end{align}
for $j=0,\ldots,n$ and $n\in \N_0$.
\begin{lem}[cf. {\cite[Lemma~3.16]{Jäger}}]\label{lem: F 2 holds}
Assume $(\I)_{n-1}$ for $n\in \N$. Suppose $f_\beta$ verifies $(\A\ref{axiom: 6})$, 
$(\mc A\ref{axiom: 10})$
and we are given $K_l,M_l$ ($l=0,\ldots,n-1$) such that $\left(\mc F1\right)_{n-1}'$, $\left(\mc F\ref{axiom: diophantine 2}\right)_{n-1}$ hold
for $\beta \in \mathscr B(n)$.
If $\sum_{j=0}^{n-1}\frac{1}{K_{j}} \leq \frac16$, then there exists $M_n\in [K_{n-1}M_{n-1},2K_{n-1}M_{n-1}]$ such that $\left(\mc F\ref{axiom: diophantine 2}\right)_n$
holds for $f_\beta$ ($\beta \in \mathscr B(n)$).
\end{lem}
\begin{proof}
By Proposition~\ref{prop: monotonicity}, monotonicity of $f_{(\cdot)}(\theta,\x)$ yields that $f_\beta$ ($\beta \in \mathscr B(n)$) verifies $(\mc F \ref{axiom: diophantine 2})_n$ if just $f_{\beta_+(n)}$ does.
Therefore, we only consider $\beta=\beta_+(n)$ and suppress the index $\beta$ in the following.

Let $j=0,\ldots,n-1$. Then,
\begin{align*}
 \mc I_n- (M_n-1)\w  \cap
\bigcup_{l=-(M_j-1)}^{M_j+1}\mc I_j+l\w \neq \emptyset
\end{align*}
implies
\begin{align*}
\mc I_j- (M_n-1)\w\cap \mc I_j+l\w \neq \emptyset,
\end{align*}
for some $l\in\{-M_j+1,-M_j+2,\ldots ,M_j+1\}$.
By $\left(\mc F1\right)_{n-1}'$,
\begin{align*}
 \#\left\{q \in [K_{n-1}M_{n-1},2K_{n-1}M_{n-1}]\cap \N \left |
\mc I_j- (q-1)\w \cap \mc I_j+l\w \neq \emptyset
\right .\right\}\leq \frac{K_{n-1}M_{n-1}}{2K_{j}M_{j}}.
\end{align*}
Hence,
\begin{align*} 
\#\left\{q \in [K_{n-1}M_{n-1},2K_{n-1}M_{n-1}]\cap \N\left |
\mc I_j- (q-1)\w \cap \bigcup_{l=-(M_j-1)}^{M_j+1} \mc I_j+l\w \neq \emptyset
\right .\right\}\leq (2M_j+1) \frac{K_{n-1}M_{n-1}}{2K_{j}M_{j}}.
\end{align*}
For the number of $q\in \{K_{n-1}M_{n-1},K_{n-1}M_{n-1}+1, \ldots, 2K_{n-1}M_{n-1}\}$ with $\mc  I_j+ (q+1)\w \cap\bigcup_{l=-(M_j-1)}^{M_j+1} \mc I_j+l\w\neq \emptyset$, we get the same upper bound.
Therefore,
\begin{align*}
 & \#\left\{q \in [K_{n-1}M_{n-1},2K_{n-1}M_{n-1}]\cap \N\left |
\left(\mc I_j- (q-1)\w \cup \mc I_j+ (q+1)\w \right)\cap \bigcup_{j=0}^{n-1}\bigcup_{l=-(M_j-1)}^{M_j+1} \mc I_j+l\w \neq \emptyset
\right .\right\}\\
& \leq 2 K_{n-1}M_{n-1} \sum_{j=0}^{n-1}\frac{2M_j+1}{2K_{j}M_{j}}\leq
3K_{n-1}M_{n-1} \sum_{j=0}^{n-1}\frac{1}{K_{j}}.
\end{align*}
Thus, if $\sum_{j=0}^{n-1}\frac{1}{K_{j}} \leq \frac16$, there is
$M_n\in [K_{n-1}M_{n-1},2K_{n-1}M_{n-1}]\cap \N$ such that 
$\left(\mc F\ref{axiom: diophantine 2}\right)_n$ holds.
\end{proof}
Given $\alpha>1$ and $b_1=1-1/K_0$, set 
\begin{align*}
 \nu\=s-c\left(\alpha,b_1^2\right)S^2\alpha^{-(2b_1^2/p-5(1-b_1^2)p)},
\end{align*}
where $c\left(\alpha,b_1^2\right)$ is as in Lemma~\ref{lem: nu n}.
Theorem~\ref{thm: sink source orbit introduction} follows from the following statement.
\begin{thm}\label{thm: sink-source orbit refined} 
Suppose $\w$ is Diophantine of type $(\mathscr C,\eta)$ and $\left(f_\beta\right)_{\beta\in [0,1]}$ satisfies $(\A\ref{axiom: 1})$-$(\A\ref{axiom: 16})$.
Let there be $p\geq \sqrt{2}$ and $\alpha>1$ with
\begin{align*}
 \alpha_c^{-1}=\alpha_e = \alpha^{\frac2{p}}, \qquad \alpha_l^{-1}=\alpha_u = \alpha^p.
\end{align*}
Further, assume $2|\I_{0,\beta}|<\mathscr C (2K_0M_0)^{-\eta}$ for some $K_0,M_0 \in \N_{\geq2}$ and assume $\nu>0$. Then there exists $\alpha_0=\alpha_0(\nu,K_0,M_0,p,|C|,|E|,\eta,\mathscr C)$ such that
if $\alpha>\alpha_0$, there is $\beta_c \in [0,1]$ 
such that $f_{\beta_c}$ has a sink-source orbit in $\T^d\times [e^-,c^+]$, and hence an SNA and an SNR.
\end{thm}
\begin{rem}
 We can choose $\alpha_0$ to depend monotonously decreasing on $\nu$.
Further, note that since we assume $\nu>0$, we necessarily have $K_0>2+5p^2$.
\end{rem}

\begin{proof}
\ref{eq: defn diophantine prime} amounts to 
$2|\I_{n,\beta}|<d(k\w,0) \ (k=1,\ldots,2K_nM_n)$.
Note that since $\w$ is Diophantine of type $(\mathscr C,\eta)$, we have
\begin{align}\label{eq: F 1 n}
 2|\I_{n,\beta}|<\mathscr C (2K_nM_n)^{-\eta}<d(k\w,0).
\end{align}
Hence, $(\mc F 1)_0'$ holds by the assumptions. Therefore, Lemma~\ref{lem: F 2 holds} together with Remark~\ref{rem: I0}
yields the existence of $M_1 \in [K_0M_0,2K_0M_0]$ such that $(\mc F \ref{axiom: diophantine 2})_1$ holds for $\beta \in \mathscr B(1)$.
Lemma~\ref{lem: nu n} gives $\nu_0^\pm,\nu_1^\pm \geq \nu>0$ such that
Lemma~\ref{lem: geometric} yields $(\I)_1$.
By means of Lemma~\ref{lem: geometric} together with Lemma~\ref{lem: height of strips}, we get
\begin{align*}
 |\I_{1,\beta}|
\leq {\mathscr C}_1 
\alpha_c^{M_0/2}=
{\mathscr C}_1 
\alpha^{-M_0/p},
\end{align*}
where ${\mathscr C}_1 \= \sqrt{8\frac{|C|+|E|}{\nu}}$.

Let us set $\left(K_n\right)_{n\in \N_0}\= \left(K_0 \kappa^n\right)_{n\in \N_0}$ for some $\kappa \in \N_{\geq 2}$ large enough to guarantee that $b>b_1^2$, and hence
$\nu_n^\pm\geq\nu$ for all $n\in \N$.
Then, since $M_1\in[K_{0}M_{0},2 K_{0}M_{0}]$,
the right hand side of (\ref{eq: F 1 n}) is bounded from below by
\begin{align*}
 \mathscr C (2K_1M_1)^{-\eta}\geq\frac{\mathscr C}{\left(4\kappa K_0^2M_0\right)^{\eta}}
\end{align*}
for $n=1$.
Therefore, if $\alpha$ is large enough,
$(\mc F 1)_1'$ is verified.

Let $n\in \N_{\geq 2}$. Suppose $(\I)_{n-1}$ and $(\mc F)_{n-1}$ (for $\beta \in \mathscr B(n)$) hold
with 
 $K_j=K_0 \kappa^j$ and $M_j\in[K_{j-1}M_{j-1},2 K_{j-1}M_{j-1}]$ for $j=1,\ldots,n-1$.
As for $n=1$, Lemma~\ref{lem: F 2 holds} yields $M_n\in [M_{n-1}K_{n-1},2M_{n-1}K_{n-1}]$ such that $(\mc F \ref{axiom: diophantine 2})_n$ holds. Now, $(\I)_n$ follows similarly as in the case $n=1$.
By means of Lemma~\ref{lem: geometric} and Lemma~\ref{lem: height of strips}, we get
\begin{align*}
2\left|\I_{n,\beta}\right |\leq 
2{\mathscr C}_1  \alpha_c^{\frac{b_{n-1}}2 M_{n-1}} \alpha_u^{\frac{1-b_{n-1}}2 M_{n-1}}\leq
2 {\mathscr C}_1 
\alpha^{\left(-\frac{b}p+\frac{1-b}{2}p\right)M_{n-1}}
\leq
\frac{\mathscr C}{\left(4\kappa^{2n-1} K_0^2M_{n-1}\right)^{\eta}}
\end{align*}
and thereby $(\mc F 1)_n'$ for $\beta \in \mathscr B(n)$,
where the last inequality holds for all $n\in \N$ if $\alpha$ is large enough.

By induction, we thus see that there are sequences $\left(M_n\right)_{n\in\N_0}$
and $\left(K_n\right)_{n\in\N_0}$ such that $(\I)_n$ is true for all $n\in \N_0$.
Moreover, with these sequences we get $(\mc F)_n$ for each $n\in \N$ and $\beta \in \mathscr B(n)$. Applying Proposition~\ref{prop: sink source orbit} finishes the proof.
\end{proof}
\bibliography{Literaturnachweis_SNA}{}
\bibliographystyle{habbrv}

\vspace*{0.2in}

\author{$\ $ \\

Gabriel Fuhrmann\\
Emmy Noether Group: Low-dimensional and Nonautonomous Dynamics,\\
TU-Dresden,\\
Dresden, Germany.\\
\texttt{Gabriel.Fuhrmann@mailbox.tu-dresden.de}
\end{document}

%% file: bibstandard.tex
\newcommand{\w}{\omega}
\newcommand{\x}{x}

\newcommand{\X}{\ensuremath{X}}
\renewcommand{\d}{\ensuremath{\partial}}

\newcommand{\mc}{\mathcal}

\newcommand{\alphlist}{\begin{list}{(\alph{enumi})}{\usecounter{enumi}}}
\newcommand{\romanlist}{\begin{list}{(\roman{enumi})}{\usecounter{enumi}}}
\newcommand{\listend}{\end{list}}

\renewcommand{\:}{\colon}
\renewcommand{\=}{\coloneqq}

\newcommand{\ssq}{\ensuremath{\subseteq}}

\newcommand{\eps}{\ensuremath{\varepsilon}}

\newcommand{\lam}{\ensuremath{\lambda}}

\newcommand{\T}{\ensuremath{\mathbb{T}}}

\newcommand{\N}{\ensuremath{\mathbb{N}}} 
\newcommand{\R}{\ensuremath{\mathbb{R}}}
\newcommand{\Z}{\ensuremath{\mathbb{Z}}}

\newcommand{\I}{\ensuremath{\mathcal I}}
\newcommand{\A}{\ensuremath{\mathcal A}}
\newcommand{\B}{\ensuremath{\mathcal B}}

\newcommand{\sgn}{\ensuremath{\mathrm{sgn}}}

